\documentclass[12pt,a4paper]{amsart}
\usepackage{amsmath}
\usepackage{amsthm}
\usepackage{amssymb}
\usepackage[latin1]{inputenc}
\usepackage{framed}
\usepackage[english]{babel}
\usepackage{verbatim}
\usepackage[shortlabels]{enumitem}
\usepackage{expdlist}
\usepackage[square,numbers]{natbib}
\usepackage{graphicx}
\usepackage{multicol}
\usepackage{tikz-cd}

\newtheoremstyle{mio}%
	{}{} 
	{\itshape}{} 
	{\bfseries}{.}{ } 
	{#1 #2\thmnote{\mdseries #3}} 
\theoremstyle{mio}
\newtheorem{teor}{Theorem}[section]
\newtheorem{cor}[teor]{Corollary}
\newtheorem{prop}[teor]{Proposition}
\newtheorem{lemma}[teor]{Lemma}
\newtheorem{defin}[teor]{Definition}

\theoremstyle{definition}
\newtheorem{ex}[teor]{Example}
\newtheorem{oss}[teor]{Remark}

\newcounter{contacasi}
\newcommand{\caso}{\vspace{0.8cm}\addtocounter{contacasi}{1}\underline{\textbf{Case \arabic{contacasi}.}}~}

\newcommand{\nondiv}{\mathcal{G}_0}

\renewcommand{\star}{\ast}
\newcommand{\insstar}{\mathrm{Star}}

\newcommand{\insstarid}{\mathcal{F}}
\newcommand{\buoni}{\mathcal{Q}}

\newcommand{\antichain}{\mathcal{A}}
\newcommand{\insanti}{\Omega}
\newcommand{\numanti}{\omega}
\newcommand{\numantiincl}{\omega_{\mathrm{i}}}
\newcommand{\dedekind}{\numanti}

\newcommand{\qm}{\mathrm{qm}}
\newcommand{\numbuoni}[1]{\insstar_{#1}}

\newcommand{\eql}{\end{equation*}\begin{equation*}}
\newcommand{\eqlnn}{\end{equation*} \begin{equation*}}
\newcommand{\inv}[1]{\frac{1}{#1}}
\newcommand{\ins}[1]{\mathbb{#1}}
\newcommand{\insN}{\ins{N}}
\newcommand{\inN}{\in\insN}
\newcommand{\insZ}{\ins{Z}}
\newcommand{\inZ}{\in\insZ}
\newcommand{\insfracid}{\mathcal{F}}

\title[]{Star operations on numerical semigroups: antichains and explicit results}
\author{Dario Spirito}
\date{}
\address{Dipartimento di Matematica, Universit\`{a} degli Studi ``Roma Tre'', Largo San Leonardo Murialdo, 1, 00146 Rome, Italy}
\email{spirito@mat.uniroma3.it}

\begin{document}

\begin{abstract}
We introduce an order on the set of non-divisorial ideals of a numerical semigroup $S$, and link antichains of this order with the star operations on $S$; subsequently, we use this order to find estimates on the number of star operations on $S$. We then use them to find an asymptotic estimate on the number of nonsymmetric numerical semigroups with $n$ or less star operations, and to determine these semigroups explicitly when $n=10$.
\end{abstract}

\maketitle

\section{Introduction}
Star operations are a class of closure operations originally defined on integral domains as a generalization of the so-called \emph{divisorial closure} (or \emph{$v$-operation}) \cite{krull_idealtheorie,gilmer}; subsequently, they have been generalized to the context of semigroups in order to generalize certain ring-theoretical properties \cite{star_semigroups}. In recent years, a subject of study has been the cardinality of the set of star operations: precise countings have been obtained for the cases of pseudo-valuation domains \cite{cardinality_pvd}, $h$-local Pr\"ufer domains \cite{twostar} and some classes of Noetherian one-dimensional domains \cite{houston_noeth-starfinite,starnoeth_resinfinito}. More generally, it has been studied when the set of star operations is finite \cite{hmp_finite}.

This paper follows the approach of the previous papers \cite{semigruppi_main} and \cite{semigruppi_mu3}, where the main problem studied was to find ways to estimate (or, if possible, to count precisely) the number of star operations on an arbitrary numerical semigroup, and to deterimine explicitly all the numerical semigroups with exactly $n$ star operations. More specifically, in \cite{semigruppi_main} it was proved that, if $n>1$, then there are only a finite number of numerical semigroups with exactly $n$ star operations, while \cite{semigruppi_mu3} provided an explicit formula for the cardinality of the set of star operations on $S$ when $S$ has multiplicity 3.

The goal of this paper is to improve the estimates proved in \cite{semigruppi_main}, with the dual objective to obtain an asymptotic bound for the number of nonsymmetric numerical semigroups with $n$ or less star operations and to determine explicitly such semigroups in the case $n=10$. This is accomplished by studying a natural order on the set of nondivisorial ideals (introduced and sketched in \cite{semigruppi_mu3}) and linking star operations with the antichains of this order; this allows to establish several inequalities between the size of $\insstar(S)$ and the invariants of $S$.

\section{Notation and basic facts}
In analogy with \cite{semigruppi_main} and \cite{semigruppi_mu3}, we shall follow the notation of \cite{fontana_maximality}. For further informations about numerical semigroups, the interested reader may consult \cite{rosales_libro}.

A \emph{numerical semigroup} is a subset $S\subseteq\insN$ such that $0\in S$, $a+b\in S$ for every $a,b\in S$ and such that $\insN\setminus S$ is finite. If $a_1,\ldots,a_n$ are natural numbers, $\langle a_1,\ldots,a_n\rangle$ denotes the semigroup generated by $a_1,\ldots,a_n$, or, more explicitly, the set $\{\lambda_1a_1+\cdots+\lambda_na_n: \lambda_i\inN\}$. The notation $S=\{0,b_1,\ldots,b_n,\rightarrow\}$ indicates that $S$ contains 0, $b_1$, \ldots, $b_n$ and all the integers bigger than $b_n$.

An \emph{ideal} $I$ of $S$ is a nonempty subset $I\subseteq S$ such that $i+s\in I$ for every $i\in I$, $s\in S$; the \emph{maximal ideal} of $S$ is $M_S:=S\setminus\{0\}$. A \emph{fractional ideal} of $S$ is an $I\subseteq\insZ$ such that $d+I$ is an ideal of $S$ for some $d\in\insZ$. We denote by $\insfracid(S)$ the set of fractional ideals of $S$, and by $\insfracid_0(S)$ the set of fractional ideals contained between $S$ and $\insN$ or, equivalently, the set of fractional ideals whose minimal element is 0. For every fractional ideal $I$, we have $I-\min(I)\in\insfracid_0(S)$. The intersection of a family of fractional ideals, if nonempty, is a fractional ideal; the union of a family of fractional ideals is a fractional ideal, provided that there is an integer smaller than every element of every ideal of the family. In particular, the union of a family of ideals contained in $\insN$ is an ideal.

The \emph{Frobenius number} $g(S)$ of a numerical semigroup $S$ is the biggest element of $\insZ\setminus S$, while the \emph{degree of singularity} of $S$, denoted by $\delta(S)$, is defined as the cardinality of $\insN\setminus S$. The \emph{multiplicity} $\mu(S)$ of $S$ is the least positive integer in $S$, i.e., the least element of $M_S$.

If $I,J$ are ideals of $S$, then $(I-J):=\{x\in\insZ: x+J\subseteq I\}$ is an ideal of $S$. The set $(S-M_S)\setminus S$ is denoted by $T(S)$, and its cardinality $t(S)$ is called the \emph{type} of $S$. For every numerical semigroup $S$, $g(S)\in T(S)$, and hence $t(S)$ is positive.

In analogy with integral domains, we define a \emph{star operation} on $S$ as a map $\star:\insfracid(S)\longrightarrow\insfracid(S)$, $I\mapsto I^\star$, such that, for any $I,J\in\insfracid(S)$, $a\in\insZ$, the following properties hold:
\begin{enumerate}[(a)]
\item $I\subseteq I^\star$;
\item if $I\subseteq J$, then $I^\star\subseteq J^\star$;
\item $(I^\star)^\star=I^\star$;
\item $a+I^\star=(a+I)^\star$;
\item $S^\star=S$.
\end{enumerate}
An ideal $I$ such that $I=I^\star$ is said to be \emph{$\star$-closed}. The set of $\star$-closed ideals is denoted by $\insstarid^\star(S)$, or $\insstarid^\star$ if $S$ is understood from the context. We indicate with $\insstar(S)$ the set of star operations of $S$; for every numerical semigroup $S$, $\insstar(S)$ is finite. If $n>1$, then there are only a finite number of numerical semigroups $S$ such that $|\insstar(S)|=n$ \cite[Theorem 4.15]{semigruppi_main}.

The set of star operations has a natural ordering, where $\star_1\leq\star_2$ if and only if $I^{\star_1}\subseteq I^{\star_2}$ for every ideal $I$ or, equivalently, if and only if $\insfracid^{\star_1}(S)\supseteq\insfracid^{\star_2}(S)$. Endowed with this ordering, the minimum of $\insstar(S)$ is the identity star operation (usually denoted by $d$), while the maximum is the star operation $I\mapsto(S-(S-I))$, usually denoted by $v$. Ideals that are $v$-closed are commonly called \emph{divisorial}. We denote by $\nondiv(S)$ the set of nondivisorial ideals $I$ such that $\min I=0$, that is, $\nondiv(S):=\insfracid_0(S)\setminus\insstarid^v(S)$.

\section{Ordering and antichains}\label{sect:ord}
Let $I$ be an ideal of $S$. Then, $I$ defines a star operation $\star_I$ such that, for every ideal $J$ of $S$,
\begin{equation}\label{eq:defstarI}
J^{\star_I}:=J^v\cap(I-(I-J))= J^v\cap\bigcap_{\alpha\in(I-J)}(-\alpha+I).
\end{equation}
(For the equivalence of the two representations, see \cite[Proposition 3.6]{semigruppi_main}.) Equivalently, $\star_I$ can be defined as the biggest star operation $\star$ such that $I$ is $\star$-closed. This definition allows to define a preorder on the set of fractional ideals.
\begin{defin}
Let $S$ be a numerical semigroup and let $I,J\in\insfracid(S)$. We say that $I$ is $\star$-minor than $J$, and we write $I\leq_\star J$, if $\star_I\geq\star_J$ or, equivalently, if $I$ is $\star_J$-closed.
\end{defin}

However, $\leq_\star$ is not an order on $\insfracid(S)$. Indeed, if $a\in\insZ$, then $a+I$ is $\star$-closed if and only if $I$ is; therefore, $\star_I=\star_{a+I}$, so that $I\leq_\star a+I$ and $a+I\leq_\star I$. Moreover, if $I$ is a divisorial ideal, then $\star_I=v$. These are the unique possibilities: that is, if $I,J$ are nondivisorial ideals and $\star_I=\star_J$, then $I=a+J$ for some $a\in\insZ$ \cite[Corollary 3.9]{semigruppi_main}. In particular, if $I,J\in\nondiv(S)$ and $I\neq J$, then $\star_I\neq\star_J$; therefore, $(\nondiv(S),\leq_\star)$ is a partially ordered set.

Let $g=g(S)$ and let $M_g:=\{a\in\insN:g-a\notin S\}=\bigcup\{I\in\insfracid_0(S):g\notin I\}$. By \cite[Corollary 4.5]{semigruppi_main} (see also \cite[Satz 4 and Hillsatz 5]{jager}), every ideal $I$ of $S$ is $\star_{M_g}$-closed; in terms of the order, this means that $M_g$ is the maximum of $(\nondiv(S),\leq_\star)$. On the other hand, $(\nondiv,\leq_\star)$ does not have (in general) a minimum, since the biggest star operation is $v$, and we are considering only operations generated by nondivisorial ideals. However, since $\nondiv$ is finite, there are always minimal elements; these are the ideals $I$ such that $\insstarid^{\star_I}=\insstarid^v\cup\{n+I: n\inZ\}$. For example, if $S=\{0,\mu,\rightarrow\}$, then every ideal in the form $I=\{0,a,\rightarrow\}$ (with $1<a<\mu$) is a minimal element of $(\nondiv,\leq_\star)$.

More generally, if $\Delta$ is a set of ideals of $S$, we can define a star operation $\star_\Delta$ as $\star_\Delta:=\inf_{I\in\Delta}\star_I$, or more explicitly as
\begin{equation}\label{eq:defstardelta}
J^{\star_\Delta}:=\bigcap_{I\in\Delta}J^{\star_I}=J^v\cap\bigcap_{I\in\Delta}(I-(I-J))= J^v\cap\bigcap_{I\in\Delta}\bigcap_{\alpha\in(I-J)}(-\alpha+I)
\end{equation}
As before, $\star_\Delta$ can also be defined as the biggest star operation $\star$ such that every element of $\Delta$ is $\star$-closed; in particular, for any star operation $\star$, we have $\star=\star_{\insstarid^\star}$, and thus this construction yields all star operations. We call $\star_\Delta$ the star operation \emph{generated} by $\Delta$. However, the order relation $\leq_\star$ cannot be easily generalized to the power set of $\nondiv(S)$, because, in general, it is possible that $\star_\Delta=\star_\Lambda$ while $\Delta\neq\Lambda$: for example, if $J$ is nondivisorial and $\star_I$-closed, then $\{I\}$ and $\{I,J\}$ define the same star operation. To avoid this problem, we introduce the following definition.
\begin{defin}\label{def:antichain}
Let $(\mathcal{P},\leq)$ be a partially ordered set. An \emph{antichain} of $\mathcal{P}$ is a set $\Delta\subseteq\mathcal{P}$ such that no two members of $\Delta$ are comparable. We denote by $\insanti(\mathcal{P})$ the set of antichains of $\mathcal{P}$, and by $\numanti(\mathcal{P})$ its cardinality.
\end{defin}

Thus, we would hope that, if $\Delta\neq\Lambda$ are antichains of $(\nondiv(S),\leq_\star)$, then $\star_\Delta\neq\star_\Lambda$. However, we will show in Example \ref{es:Anonsurj} that this is not always true; before showing the example, we need some notation.

We denote by $\antichain$ and $\star$ the two maps
\begin{equation*}
\begin{aligned}
\antichain\colon\insstar(S) & \longrightarrow\insanti(\nondiv(S))\\
\star & \longmapsto\max\phantom{}_\star(\insfracid^\star\cap\nondiv),
\end{aligned}
\end{equation*}
(where $\max\phantom{}_\star(\Delta)$ indicates the maximal elements of $\Delta$ in the $\star$-order) and
\begin{equation*}
\begin{aligned}
\star\colon\insanti(\nondiv(S)) & \longrightarrow \insstar(S)\\
\Delta & \longmapsto \star_\Delta.
\end{aligned}
\end{equation*}

Note that, if $I\in\antichain(\star)$ and $J\leq_\star I$, then $J$ is $\star_I$-closed, and thus $\star$-closed; therefore, since $\insstarid^\star$ uniquely determines $\star$, the set $\antichain(\star)$ uniquely determines $\star$, and thus $\antichain$ is injective. Moreover, it is clear that $\star_{\antichain(\star_\Delta)}=\star_\Delta$ for every $\Delta\subseteq\nondiv(S)$; therefore, $\star\circ\antichain$ is the identity on $\insstar(S)$, and $\star$ is a surjective map. In particular, $|\insstar(S)|\leq\numanti(\nondiv(S))$. Note also that $\numanti(\nondiv)$ is finite, because $\nondiv$ is finite.

If $\Delta=\emptyset$, then $\star_\emptyset=v$; if $\Delta=\{I\}$ is a single ideal, then $\insstarid^{\star_I}=\insstarid^v\cup\{J\in\nondiv(S): J\leq_\star I\}$ and thus $\antichain(\star_I)=\{I\}$. With this terminology, asking if $\star_\Delta\neq\star_\Lambda$ whenever $\Delta\neq\Lambda$ are antichains of $\nondiv(S)$ amounts to asking if $\antichain$ is a surjective map, or, equivalently, if $\antichain\circ\star$ is the identity on $\insanti(\nondiv(S))$. The answer is in general negative, as the following example shows.

\begin{ex}\label{es:Anonsurj}
Let $S:=\langle 5,6,7,8,9\rangle=\{0,5,\rightarrow\}$, $I:=S\cup\{3,4\}$, $J:=S\cup\{1,3\}$, $L:=S\cup\{4\}$. Calculations show that $\Delta:=\{I,J\}$ is an antichain of $\nondiv$, and that $L^{\star_I}=L\cup\{3\}=I$, $L^{\star_J}=L\cup\{2\}$, so that $L$ is nor $\star_I$ nor $\star_J$-closed. However,
\begin{equation*}
L^{\star_\Delta}=L^{\star_I}\cap L^{\star_J}=L
\end{equation*}
and hence $\antichain(\star_\Delta)$ must contain an ideal $\geq_\star L$. Therefore, $\antichain\circ\star(\Delta)\neq\Delta$, i.e., $\antichain\circ\star$ is not the identity on $\insanti(\nondiv(S))$ (and actually $\Delta\neq\antichain(\star)$ for every $\star\in\insstar(S)$).
\end{ex}

\section{Prime star operations and atoms}
\begin{defin}
A star operation $\star$ is \emph{prime} if, whenever $\star\geq\star_1\wedge\star_2$, we have $\star\geq\star_1$ or $\star\geq\star_2$.
\end{defin}

\begin{prop}
A prime star operation is principal, i.e., $\star=\star_I$ for some ideal $I$.
\end{prop}
\begin{proof}
Suppose it is not, and consider the antichain $\antichain(\star):=\{I_1,\ldots,I_n\}$. Then, $\star=\star_{I_1}\wedge\cdots\wedge\star_{I_n}$, and in particular $\star\leq\star_{I_i}$ for every $i\in\{1,\ldots,n\}$.

However, an inductive argument applied to the definition of prime star operation shows that $\star\geq\star_I$ for some $I\in\antichain{\star}$; hence, $\star_I\leq\star\leq\star_I$, and $\star=\star_I$, that is, $\star$ is a principal star operation.
\end{proof}

\begin{defin}
If $I\in\insfracid_0(S)$ is an ideal of $S$ such that $\star_I$ is prime, we say that $I$ is an \emph{atom} of $\nondiv(S)$.
\end{defin}
Note that every divisorial ideal $I\in\insfracid_0(S)$ is an atom, since $\star_I=v$ is prime.

\begin{prop}\label{prop:atom}
Let $S$ be a numerical semigroup and $I\in\nondiv(S)$. The following are equivalent:
\begin{enumerate}[(i)]
\item $I$ is an atom of $\nondiv(S)$;
\item for every $\star_1,\star_2\in\insstar(S)$, $I$ is $\star_1\wedge\star_2$-closed if and only if $I$ is $\star_1$- or $\star_2$-closed;
\item for every $J_1,J_2\in\insfracid_0(S)$ such that $\star_I\geq\star_{J_1}\wedge\star_{J_2}$, we have $\star_I\geq\star_{J_1}$ or $\star_I\geq\star_{J_2}$;
\item\label{prop:atom:intersez} if $I=J_1\cap J_2$, then $I$ is $\star_{J_1}$- or $\star_{J_2}$-closed;
\item for every $\star_1,\ldots\star_n\in\insstar(S)$, $I$ is $\star_1\wedge\cdots\wedge\star_n$-closed if and only if $I$ is $\star_i$-closed for some $i\in\{1,\ldots,n\}$;
\item\label{prop:atom:vi} for every $\Delta\subseteq\insfracid(S)$, $I=I^{\star_\Delta}$ if and only if $I\leq_\star J$ for some $J\in\Delta$.
\end{enumerate}
\end{prop}
\begin{proof}
(ii) is just a restatement of the definition of atom, so it is equivalent to (i). Clearly (ii $\Longrightarrow$ iii), while (iii $\Longrightarrow$ iv) since if $I=J_1\cap J_2$ then $\star_I\geq\star_{J_1}\wedge\star_{J_2}$. Suppose (iv) holds and suppose that $I$ is $\star_1\wedge\star_2$-closed. Then, $I=I^{\star_1\wedge\star_2}=I^{\star_1}\cap I^{\star_2}$, and thus, if $J_i:=I^{\star_i}$, then $I$ is $\star_{J_1}$- or $\star_{J_2}$-closed. However, $\star_{J_i}\geq\star_i$, and thus $I$ is $\star_1$- or $\star_2$-closed. Hence, (iv $\Longrightarrow$ ii).

(ii $\Longrightarrow$ v) follows by induction; to show (v $\Longrightarrow$ vi), we can suppose $\Delta\subseteq\insfracid_0(S)$; since $\insfracid_0(S)$ is finite, so is $\Delta$. Hence, since $\star_\Delta=\inf_{J\in\Delta}\star_J$, if $I=I^{\star_\Delta}$ then $I$ is $\star_J$-closed for some $J\in\Delta$.

(vi $\Longrightarrow$ i) Suppose $\star_I\geq\star_1\wedge\star_2$, and let $\Delta_1:=\{J\in\nondiv(S): J=J^{\star_1}\}$, $\Delta_2:=\{J\in\nondiv(S): J=J^{\star_2}\}$, $\Delta:=\Delta_1\cup\Delta_2$. Then $I=I^{\star_\Delta}$, and thus $I\leq_\star J$ for some $J\in\Delta$: if $J\in\Delta_1$ (say), then $\star_I\geq\star_1$, and $I$ is an atom.
\end{proof}

\begin{cor}\label{cor:antcompl-diffanti}
Let $S$ be a numerical semigroup and $\Gamma\subseteq\nondiv(S)$ a set of atoms of $\nondiv(S)$. If $\Delta\neq\Lambda$ are nonempty antichains of $\Gamma$, then $\star_\Delta\neq\star_\Lambda$.
\end{cor}
\begin{proof}
Suppose $\star_\Delta=\star_\Lambda$; without loss of generality, there is a $L\in\Lambda\setminus\Delta$. Then, $L=L^{\star_\Delta}$; since $L$ is an atom, by Proposition \ref{prop:atom}\ref{prop:atom:vi} there is a $J\in\Delta$ such that $L\leq_\star J$.

Since $J=J^{\star_\Lambda}$, with the same reasoning we obtain a $L_1\in\Lambda$ such that $J\leq_\star L_1$; therefore, $L\leq_\star L_1$. Since $\Lambda$ is an antichain, with respect to the $\star$-order, we must have $L=L_1$, and thus $L=J$. But $J\in\Delta$ while $L\notin\Delta$; this is a contradiction, and $\star_\Delta\neq\star_\Lambda$.
\end{proof}

\begin{cor}\label{cor:omegagamma}
Let $S$ be a numerical semigroup and $\Gamma\subseteq\nondiv(S)$ be the set of atoms of $\nondiv(S)$. Then, $|\insstar(S)|\geq\numanti(\Gamma)$.
\end{cor}
\begin{proof}
Apply Corollary \ref{cor:antcompl-diffanti}: every nonempty antichain generates a different star operation, and the empty antichain generates the $v$-operation.
\end{proof}

Thus, a way to estimate $|\insstar(S)|$ is through finding atoms. The next proposition estabilishes a useful criterion.

\begin{prop}\label{prop:comparabili}
Let $S$ be a numerical semigroup and $I\in\nondiv(S)$.
\begin{enumerate}[(a)]
\item\label{prop:comparabili:->atom} If, for every $\star_1,\star_2\in\insstar(S)$, we have $I^{\star_1}\subseteq I^{\star_2}$ or $I^{\star_2}\subseteq I^{\star_1}$, then $I$ is an atom. 
\item\label{prop:comparabili:atom->} If $I^\star$ is an atom for every $\star\in\insstar(S)$, then $I^{\star_1}$ and $I^{\star_2}$ are comparable for every pair $\star_1,\star_2$ of star operations.
\end{enumerate}
\end{prop}
\begin{proof}
(a) Suppose $I$ is not an atom. Then, there are star operations $\star_1,\star_2$ such that $\star_I\geq\star_1\wedge\star_2$ but $\star_I\ngeq\star_1$ and $\star_I\ngeq\star_2$. Then, $I\neq I^{\star_1}$ and $I\neq I^{\star_2}$, but $I=I^{\star_1\wedge\star_2}=I^{\star_1}\cap I^{\star_2}$, so that $I^{\star_1}$ and $I^{\star_2}$ are not comparable.

(b) If $I^{\star_1}$ and $I^{\star_2}$ are not comparable, let $J:=I^{\star_1}\cap I^{\star_2}=I^{\star_1\wedge\star_2}$. Then, $I^{\star_i}\subseteq J^{\star_i}\subseteq(I^{\star_i})^{\star_i}=I^{\star_i}$, and thus $I^{\star_i}=J^{\star_i}=:J_i$. By hypothesis, $J$ is an atom; by Proposition \ref{prop:atom}\ref{prop:atom:intersez}, $J$ is $\star_{J_i}$-closed for some $i$ (say $i=1$). Then, since $J_1$ is $\star_1$-closed, we have $\star_1\leq\star_{J_1}$ and
\begin{equation*}
J_1=J^{\star_1}\subseteq J^{\star_{J_1}}=J,
\end{equation*}
and thus $J=J_1$. In particular, $J_1\subseteq J_2$, and $I^{\star_1}$ and $I^{\star_2}$ are comparable.
\end{proof}

A result similar to the next result will be Proposition \ref{prop:atom-MaI}.
\begin{prop}\label{prop:anticatene-V1}
Let $S$ be a numerical semigroup and $I\in\insfracid_0(S)$. If $|I^v\setminus I|=1$, then $I$ is an atom of $\nondiv(S)$.
\end{prop}
\begin{proof}
Immediate from Proposition \ref{prop:comparabili}\ref{prop:comparabili:->atom}, since $I^\star$ is contained between $I$ and $I^v$, and there are no ideals properly inbetween.
\end{proof}

\begin{prop}\label{prop:ufd}
Let $S$ be a numerical semigroup. The following are equivalent:
\begin{enumerate}[(i)]
\item every ideal of $S$ in $\insfracid_0(S)$ is an atom;
\item for every ideal $I$ and every $\star_1,\star_2\in\insstar(S)$, the ideals $I^{\star_1}$ and $I^{\star_2}$ are comparable;
\item the map $\antichain:\insstar(S)\longrightarrow\insanti(\nondiv(S))$, $\star\mapsto\antichain(\star)$, is bijective;
\item $\antichain\circ\star$ is the identity on $\insanti(\nondiv(S))$;
\item for every antichain $\Delta$ of $\nondiv(S)$, $\antichain(\star_\Delta)=\Delta$;
\item $|\insstar(S)|=\numanti(\nondiv(S))$.
\end{enumerate}
\end{prop}
\begin{proof}
(i $\Longrightarrow$ ii) follows from Proposition \ref{prop:comparabili}\ref{prop:comparabili:atom->}, since each $I^\star$ is an atom; (ii $\Longrightarrow$ i) is a direct consequence of Proposition \ref{prop:comparabili}\ref{prop:comparabili:->atom}.

(i $\Longrightarrow$ iii) Since $\antichain$ is injective, it is enough to show that it is surjective. Let $\Delta$ be a nonempty antichain of $\nondiv(S)$, and consider the star operation $\star_\Delta$: if $\antichain(\star_\Delta)=\Lambda\neq\Delta$, then $\star_\Lambda=\star_\Delta$, against Corollary \ref{cor:antcompl-diffanti}.

(iii $\iff$ iv $\iff$ v) follows from the discussion after Definition \ref{def:antichain}.

(iv $\Longrightarrow$ i) Suppose $I\in\insfracid_0(S)$ is not an atom: then $I$ is not divisorial, and there are ideals $J_1,J_2$ such that $I=J_1\cap J_2$ but $I$ is not $\star_{J_1}$- nor $\star_{J_2}$-closed. The ideals $J_1$ and $J_2$ are not $\star$-comparable: if $J_1\leq_\star J_2$ (say), then $J_1=J_1^{\star_{J_2}}$ and thus $I$ would be $\star_{J_2}$-closed, which is impossible. Hence, $\Delta:=\{J_1,J_2\}$ is an antichain, and thus $\antichain(\star_{\Delta})=\Delta$ (since iv $\iff$ v).

Since $I$ is $\star_\Delta$-closed, $\star_\Delta=\star_\Delta\wedge\star_I=\star_{\Delta\cup\{I\}}$, and thus $\Delta\cup\{I\}$ cannot be an antichain. However, $I$ is not $\star$-minor than each $J_i$, and thus $I\geq_\star J_i$ for some $i$. This would imply that $J_i$ is not $\star$-maximal in $\insstarid^{\star_\Delta}$, that is, $J_i\notin\antichain(\star_\Delta)$, a contradiction; therefore, $I$ is an atom.

(iii $\iff$ vi) is a simple consequence of the finiteness of $\insstar(S)$ and $\insanti(\nondiv(S))$.
\end{proof}

\section{The sets $\buoni_a$}\label{sect:Qa}
Probably the most important property of prime star operation is expressed in Corollary \ref{cor:antcompl-diffanti}: different antichains, composed of atoms, generate different star operations. The goal of this section is to determine other sets enjoying this property.

\begin{defin}
Let $S$ be a numerical semigroup. For every $a\inN\setminus S$, let $\buoni_a(S):=\{I\in\insfracid_0(S): a=\sup(\insN\setminus I), a\in I^v\}$.
\end{defin}

For every $a\inN\setminus S$, we define $M_a$ as
\begin{equation*}
M_a:=\{x\in\insN:a-x\notin S\}=\bigcup\{I\in\insfracid_0(S):a\notin I\}.
\end{equation*}
or, equivalently, as the biggest ideal in $\insfracid_0(S)$ that does not contain $a$ \cite[Definition 4.1 and Lemma 4.2]{semigruppi_main}.

\begin{prop}\label{prop:maxbuonia}
Let $S$ be a numerical semigroup and $\buoni_a:=\buoni_a(S)$. Then:
\begin{enumerate}[(a)]
\item\label{prop:maxbuonia:empty} $\buoni_a$ is nonempty if and only if $M_a$ is not divisorial;
\item\label{prop:maxbuonia:max} if $\buoni_a$ is nonempty, $M_a$ is its $\star$-maximum;
\item\label{prop:maxbuonia:bleqa-M} if $b\leq a$, then $M_b\leq_\star M_a$;
\item\label{prop:maxbuonia:bleqa-empty} if $\buoni_a=\emptyset$ then $\buoni_b=\emptyset$ for every $b\leq a$;
\item\label{prop:maxbuonia:buchi} if $a,g-a\notin S$, then $\buoni_a\neq\emptyset$.
\end{enumerate}
\end{prop}
\begin{proof}
\ref{prop:maxbuonia:empty} If $M_a$ is not divisorial, $a\in M_a^v$ (by virtue of the maximality of $M_a$), and thus $M_a\in\buoni_a$. Conversely, if $M_a$ is divisorial, let $I\in\insfracid_0(S)$ be an ideal such that $a\notin I$. Then, $I\subseteq M_a$, and thus $I^v\subseteq M_a^v=M_a$, and in particular $a\notin I^v$. Hence, $I\notin\buoni_a$, which therefore must be empty.

\ref{prop:maxbuonia:max} follows from noting that $I=\bigcap_{b\in\insN\setminus I} M_b$, and that each $M_b$ is $\star_{M_a}$-closed when $b\leq a$; \ref{prop:maxbuonia:bleqa-M} follows from the equality $M_b=(b-a+M_a)\cap\insN$ \cite[Lemma 4.2]{semigruppi_main}.

\ref{prop:maxbuonia:bleqa-empty} If $\buoni_a=\emptyset$, then $M_a$ is divisorial, and $\star_{M_b}\geq\star_{M_a}=v$. Thus, $\star_{M_b}=v$, $M_b$ is divisorial and $\buoni_b=\emptyset$ by point \ref{prop:maxbuonia:empty}.

Finally, \ref{prop:maxbuonia:buchi} follows from \cite[Lemma 4.7]{semigruppi_main}.
\end{proof}

A numerical semigroup $S$ is said to be \emph{symmetric} if $g-a\in S$ for every $a\in\insN\setminus S$. By \cite[Proposition 2]{fgh_semigruppi}, $S$ is symmetric if and only if $t(S)=1$, and by \cite[Proposition I.1.15]{fontana_maximality} this happens if and only if every ideal of $S$ is divisorial (equivalently, if and only if $|\insstar(S)|=1$).

If $a\in T(S)$ and $S$ is not symmetric, then $a\in M_a^v$, and thus $\buoni_a\neq\emptyset$.

\begin{prop}\label{prop:atom-MaI}
Let $S$ be a numerical semigroup, and suppose that $I\in\buoni_a$. If $|M_a\setminus I|\leq 1$, then $I$ is an atom of $\nondiv(S)$.
\end{prop}
\begin{proof}
Suppose $I=J_1\cap J_2$. Since $a\notin I$, without loss of generality we can suppose $a\notin J_1$; moreover, if $b>a$ then $b\in I$, and so $b\in J_1$. Therefore, $I\subseteq J_1\subseteq M_a$, and since $|M_a\setminus I|\leq 1$ we have $J_1=I$ or $J_1=M_a$. In the former case $I$ is trivially $\star_{J_1}$-closed; in the latter, we have $I\leq_\star M_a$ by Proposition \ref{prop:maxbuonia}\ref{prop:maxbuonia:max}, and thus $I$ is again $\star_{J_1}$-closed. The claim follows applying condition \ref{prop:atom:intersez} of Proposition \ref{prop:atom}.
\end{proof}

When $|M_a\setminus I|\geq 2$, even if $I\in\buoni_a$, it is possible that $\star_I$ is not prime. We digress to estabilish a general lemma.
\begin{lemma}\label{lemma:TstarU}
Let $S,U$ be numerical semigroups, and $I$ be an ideal of $S$ such that $S\subseteq I\subseteq U$; let $v$ be the divisorial closure of the $S$-ideals. Then, $I^{\star_U}=I^v\cap U$.
\end{lemma}
\begin{proof}
Suppose $I\subseteq-\alpha+U$. Then, $\alpha\in U$; however, since $U$ is a semigroup, $U=(U-U)$, and thus $U\subseteq-\alpha+U$. Therefore,
\begin{equation*}
I^{\star_U}=I^v\cap\bigcap_{\alpha\in(U-I)}(-\alpha+U)\supseteq I^v\cap U.
\end{equation*}
Since $I^{\star_U}\subseteq U^{\star_U}=U$, we have $I^{\star_U}\subseteq U\cap I^v$, and thus the two sides are equal.
\end{proof}

\begin{ex}
Consider the semigroup $S:=\langle 4,6,7,9\rangle=\{0,4,6,\rightarrow\}$, and let $I:=S\cup\{5\}$. Then, $I$ is a semigroup and $I^v=(S-M)=S\cup\{2,3,5\}$; in particular, $I\in\buoni_3$. Let $J_1:=I\cup\{2\}$ and $J_2:=I\cup\{3\}$: both $J_1$ and $J_2$ are semigroups containing $I$, so that $I^{\star_{J_i}}=J_i$, and in particular $I$ is not $\star_{J_1}$- nor $\star_{J_2}$-closed. However, $J_1\cap J_2=I$, and thus $I$ is $(\star_{J_1}\wedge\star_{J_2})$-closed. Hence, $I$ is not an atom of $S$.
\end{ex}	

This example could be generalized.
\begin{cor}\label{cor:g<2mu}
Let $S$ be a numerical semigroup, $t:=t(S)$, $\mu:=\mu(S)$, $g:=g(S)$; suppose $t\geq 3$ and $g\leq 2\mu-2$. Then, $S\cup\{g\}$ is an atom of $S$ if and only if $S=\langle 4,5,6,7\rangle$.
\end{cor}
\begin{proof}
If $S=\langle 4,5,6,7\rangle$, then $M_2=S\cup\{1,3\}$, and thus $S\cup\{g\}=S\cup\{3\}$  is an atom by Proposition \ref{prop:atom-MaI} (see Example \ref{es:4567} for a deeper analysis of this semigroup).

Suppose $S\neq\langle 4,5,6,7\rangle$, and let $I:=S\cup\{g\}$. Since $\mu>t\geq 3$, we have $\mu\geq 4$. If $g<\mu$ (i.e., $S=\{0,\mu,\rightarrow\}$ and $g=\mu-1$), consider the ideals $T_2:=S\cup\{\mu-1,\mu-2\}$ and $T_3:=S\cup\{\mu-1,\mu-3\}$: since $S\neq\langle 4,5,6,7\rangle$, $\mu>4$, so that $2(\mu-3)\geq\mu-1$ and both $T_2$ and $T_3$ are semigroups. By Lemma \ref{lemma:TstarU}, $I^{\star_{T_i}}=I^v\cap T_i=\insN\cap T_i=T_i$, while $I=T_2\cap T_3$; by Proposition \ref{prop:atom}, $I$ is not an atom of $S$.

Suppose $\mu<g<2\mu-2$. Then, $\mu-1,\mu-2\in T(S)$; let $T_1:=S\cup\{g,\mu-1\}$ and $T_2:=S\cup\{g,\mu-2,2\mu-4\}$. Then, both $T_1$ and $T_2$ are semigroups, $T_1\cap T_2=I$ but $I^{\star_{T_i}}=T_i\cap(S-M)$ contains $\mu-i$ and thus it is different from $I$. Hence, $I$ is not an atom of $S$.

Suppose $g=2\mu-2$. If $\{\mu+1,\ldots,2\mu-3\}\subseteq S$, then $T(S)=\{g,\mu-1\}$, and thus $t=2$; therefore, under our hypothesis, there is a $\tau\in\{\mu+1,\ldots,2\mu-3\}\setminus S$. Then, $\tau\in T(S)$ and $2\tau>g$, and thus $T_1:=S\cup\{g,\tau\}$ is a semigroup contained in $S\cup T(S)=(S-M)$, and the same happens for $T_2:=S\cup\{\mu-1,g\}$. Again, $I=T_1\cap T_2$ but $I^{\star_{T_i}}=T_i$, so that $I$ is not an atom of $S$.
\end{proof}

We resume the analysis of the $\star$-order on $\buoni_a$.
\begin{prop}\label{prop:buoni-antichain}
Let $S$ be a numerical semigroup and $\buoni_a:=\buoni_a(S)$. Let $I,J\in\buoni_a$ and $\Delta\subseteq\buoni_a$.
\begin{enumerate}[(a)]
\item\label{prop:buoni-antichain:a} If $I\nsubseteq J$ then $a\in I^{\star_J}$.
\item\label{prop:buoni-antichain:b} If $I\nsubseteq J$ for every $J\in\Delta$ then $a\in I^{\star_\Delta}$.
\item\label{prop:buoni-antichain:c} The $\star$-order on $\buoni_a$ is coarser than the inclusion, i.e., if $I\leq_\star J$ then $I\subseteq J$.
\item\label{prop:buoni-antichain:d} Let $\Delta\neq\Lambda$ be two nonempty subsets of $\buoni_a$ that are antichains with respect to inclusion. Then, $\star_\Delta\neq\star_\Lambda$.
\end{enumerate}
\end{prop}
\begin{proof}
\ref{prop:buoni-antichain:a} By definition,
\begin{equation*}
I^{\star_J}=I^v\cap\bigcap_{\gamma\in(J-I)}(-\gamma+J).
\end{equation*}
If $I\nsubseteq J$, then $0\notin(J-I)$. Thus, for each $\gamma\in(J-I)$, $a\in-\gamma+J$ and, since $I\in\buoni_a$, $a\in I^v$. Therefore, $a\in I^{\star_J}$.

\ref{prop:buoni-antichain:b} is immediate from the above point, since $I^{\star_\Delta}=\bigcap_{J\in\Delta}I^{\star_J}$; \ref{prop:buoni-antichain:c} is just a reformulation of point \ref{prop:buoni-antichain:a}.

\ref{prop:buoni-antichain:d} Suppose $\star_\Delta=\star_\Lambda$; without loss of generality there is a $I\in\Delta\setminus\Lambda$. If $I\nsubseteq J$ for every $J\in\Lambda$, then $a\in I^{\star_\Lambda}$, which is different from $I=I^{\star_\Delta}$. Otherwise, let $J\in\Lambda$ such that $J\supseteq I$. Similarly, if there is no $I'\in\Delta$ containing $J$, then $a\in J^{\star_\Delta}$, which is different from $J=J^{\star_\Lambda}$. Thus, we have $I\subseteq J\subseteq I'$ for some $I'\in\Delta$. Since $\Delta$ is an antichain with respect to the containment, we must have $I=I'$, and thus $I=J$. But this is impossible, since $I\notin\Lambda$.
\end{proof}

\begin{oss}
Note that the $\star$-order on $\buoni_a$ may really be different from the containment: for example, consider $S:=\{0,5,\rightarrow\}$ and let $I:=S\cup\{1\}$, $J:=S\cup\{1,3\}$. Both $I$ and $J$ are in $\buoni_4$, and $I\subseteq J$; we claim that $I\not\leq_\star J$.

Indeed, $I^v=\insN$; suppose $I\subseteq-\gamma+J$. Then, $\gamma\in J$, and thus $\gamma\in\{0,1,3\}$ or $\gamma\geq 5$. If $\gamma=1$ or $\gamma=3$, then $1\notin-\gamma+J$; but if $\gamma\geq 5$, then $\insN\subseteq-\gamma+J$. It follows that $I^{\star_J}=\insN\cap J=J\neq I$.
\end{oss}

We shall denote by $\numantiincl(\buoni_a)$ the number of antichains of $(\buoni_a,\subseteq)$, that is, the number of antichains of $\buoni_a$ with respect to inclusion.

When $\mathcal{P}$ is the power set $\mathcal{P}(\{1,\ldots,n\})$ of the finite set with $n$ elements, ordered by inclusion, we denote the number of antichains of $\mathcal{P}$ simply as $\dedekind(n)$. These numbers are called \emph{Dedekind numbers}; their sequence grows super-exponentially, since each family of subsets of $\{1,\ldots,n\}$ of size $\lfloor n/2\rfloor$ is an antichain. More precisely, $\dedekind(n)$ is bounded as follows (see \cite{kleitman_dedprob_II}):
\begin{equation*}
\binom{n}{\lfloor n/2\rfloor}\leq \log_2 \dedekind(n)\leq\binom{n}{\lfloor n/2\rfloor}\left(1+O\left(\frac{\log n}{n}\right)\right).
\end{equation*}
If $n$ is small, $\dedekind(n)$ can be calculated by hand: if $n=0$, then the antichains of $\mathcal{P}(\emptyset)$ are the empty antichain and the antichain $\{\emptyset\}$ composed of the only empty set. If $n=1$, then $\mathcal{P}(\{1\})=\{\emptyset,\{1\}\}$, and thus the antichains are the empty antichain, $\{\emptyset\}$ and the one formed by the set $\{1\}$. If $n=2$, then we have the empty antichain, $\{\emptyset\}$, $\{\{1\}\}$,  $\{\{2\}\}$,  $\{\{1\},\{2\}\}$ and $\{\{1,2\}\}$. Hence, $\dedekind(0)=2$, $\dedekind(1)=3$ and $\dedekind(2)=6$.

\begin{cor}\label{cor:dedekindt}
Let $S$ be a numerical semigroup and $t=t(S)$. Then, $|\insstar(S)|\geq \dedekind(t-1)-1$.
\end{cor}

Compare the similar Corollary 4.10 of \cite{semigruppi_main}, where it was proved the bound $|\insstar(S)|\geq 2^t-1$.
\begin{proof}
Consider the ideals $I_A:=S\cup A$, with $A\subseteq T(S)\setminus\{g\}$. If $A\neq\emptyset$, then $I_A\neq S$, and so $T(S)\subseteq I_A^v$; it follows that, in this case, $I_A\in\buoni_g$. By Proposition \ref{prop:buoni-antichain}\ref{prop:buoni-antichain:d}, each nonempty antichain (with respect to inclusion) of $\{I_A:A\subseteq T(S)\setminus\{g\},A\neq\emptyset\}$ generates a different star operation; however, the inclusion order is nothing but the order of the power set of $T(S)\setminus\{g\}$, which has $\dedekind(t-1)$ antichains. We must exclude the empty antichain and the antichain corresponding to the empty set, so that we have $\dedekind(t-1)-2$ star operations. Moreover, each of these operations is different from the $v$-operation, and thus $|\insstar(S)|\geq \dedekind(t-1)-1$.
\end{proof}

We cannot go much further by considering each $\buoni_a$ separetely; to obtain better estimates, we must compare star operations generated by ideals in different $\buoni_a$.
\begin{lemma}\label{lemma:buoni-compbase}
Let $S$ be a numerical semigroup, and let $I,J\in\nondiv(S)$ such that $J\leq_\star I$. If $I\in\buoni_a$ and $J\in\buoni_b$, then $a\geq b$.
\end{lemma}
\begin{proof}
The proof is the same as the proof of Proposition \ref{prop:buoni-antichain}\ref{prop:buoni-antichain:a}: if $a<b$, then $b$ belongs to both $J^v$ and $-\alpha+I$ (for every $\alpha\in I-J$), and so $b\in J^{\star_I}$, and in particular $J\neq J^{\star_I}$, against the hypothesis $J\leq_\star I$.
\end{proof}

The following is a generalization of Proposition \ref{prop:buoni-antichain}\ref{prop:buoni-antichain:d}.
\begin{prop}\label{prop:buonia-buonib}
Let $S$ be a numerical semigroup. Let $\Delta\subseteq\buoni_a$, $\Lambda\subseteq\buoni_b$ two nonempty sets which are antichains with respect to inclusion. If $\Delta\neq\Lambda$ (in particular, if $a\neq b$) then $\star_\Delta\neq\star_\Lambda$.
\end{prop}
\begin{proof}
The case $a=b$ is just Proposition \ref{prop:buoni-antichain}. Suppose (without loss of generality) that $a>b$.

Let $I\in\Delta$, and let $\gamma\inN$, $J\in\Lambda$ such that $I\subseteq-\gamma+J$. Since $\gamma+a\geq a>b$, we have $a\in-\gamma+J$, and thus $a\in I^{\star_\Lambda}\setminus I$, and $I^{\star_\Lambda}\neq I=I^{\star_\Delta}$.
\end{proof}

\begin{cor}\label{cor:sumbuonia}
Let $S$ be a numerical semigroup. Then,
\begin{equation*}
|\insstar(S)|\geq 1+\sum_{a\inN\setminus S}(\numantiincl(\buoni_a)-1)\geq 1+\sum_{a\inN\setminus S}|\buoni_a|.
\end{equation*}
\end{cor}
\begin{proof}
It is enough to apply Proposition \ref{prop:buonia-buonib} to the nonempty antichains of the $\buoni_a$, and then add the $v$-operation. For the second inequality, note that every ideal of $\buoni_a$ is an antichain of $\buoni_a$ (in every order).
\end{proof}

We can also prove a limited form of the above results for ``mixed'' antichains, i.e., antichains whose elements come from different $\buoni_a.$

\begin{prop}\label{prop:xyconsecutivi}
Let $S$ be a numerical semigroup, and let $x<y$ be two positive integers such that:
\begin{enumerate}
\item $x,y\notin S$;
\item every integer $w$ such that $x<w<y$ is in $S$;
\item $M_x$ and $M_y$ are not divisorial.
\end{enumerate}
Let $\Lambda,\Delta$ be nonempty subsets of $\buoni_y$ that are antichains with respect to inclusion, and suppose $M_y\notin\Lambda$. Then:
\begin{enumerate}[(a)]
\item\label{prop:xyconsecutivi:cupvsnocup} $\star_{\Lambda\cup\{M_x\}}\neq\star_\Delta$;
\item\label{prop:xyconsecutivi:cupvscup} if $\Lambda\neq\Delta$ then $\star_{\Lambda\cup\{M_x\}}\neq\star_{\Delta\cup\{M_x\}}$.
\end{enumerate}
\end{prop}
\begin{proof}
\emph{Claim 1}: $y-x$ is the minimal element of $M_y\setminus\{0\}$.

Indeed, $y-x\in M_y$ because $y-(y-x)=x\notin S$; on the other hand, if $0<\beta<y-x$, then $y>y-\beta>y-(y-x)=x$, and thus, by hypothesis, $y-\beta\in S$, so that $\beta\notin M_y$.

\emph{Claim 2}: Let $I\in\buoni_y\setminus\{M_y\}$. Then, $x\in M_x^{\star_I}$.

Suppose $x\notin M_x^{\star_I}$. Then, there is an $\alpha$ such that $M_x\subseteq-\alpha+I$ while $x\notin-\alpha+I$. We distinguish four cases:
\begin{enumerate}
\item $\alpha=0$: then, $M_x\subseteq I$, against the fact that $y\in M_x\setminus I$;

\item $0<\alpha<y-x$: then, $x<x+\alpha<y$; however, $x+\alpha\in S\subseteq I$, contradicting $x+\alpha\notin I$;

\item $\alpha>y-x$: then, $x$ would be contained in $-\alpha+I$, since $I$ contains each element bigger than $y$, but this is absurd;

\item $\alpha=y-x$: in this case, 
\begin{equation*}
x=\sup(\insN\setminus(-\alpha+I))=\sup(\insN\setminus((-\alpha+I)\cap\insN)),
\end{equation*}
so that $(-\alpha+I)\cap\insN\subseteq M_x$; since $M_x\subseteq-\alpha+I$, it follows that $(-\alpha+I)\cap\insN=M_x=(-\alpha+M_y)\cap\insN$. Since $I\neq M_y$, there is a $\beta\in M_y\setminus I$; if $\beta>\alpha$, then 
\begin{equation*}
-\alpha+\beta\in[(-\alpha+M_y)\cap\insN]\setminus[(-\alpha+I)\cap\insN],
\end{equation*}
against the hypothesis. Thus $\alpha>\beta$; this means that $y>y-\beta>y-\alpha=x$, and thus $y-\beta\in S$. But this contradicts the fact that $\beta\in M_y$ while $y\notin M_y$.
\end{enumerate}

We are now ready to prove \ref{prop:xyconsecutivi:cupvsnocup}. Since $\Lambda$ is a nonempty antichain of $\buoni_y\setminus\{M_y\}$, we have $x\in M_x^{\star_\Lambda}=\bigcap_{I\in\Lambda}M_x^{\star_I}$. If $\Delta$ does not contain $M_y$, then by Claim 2 we have $x\in M_x^{\star_\Delta}$, while $M_x^{\star_{\Lambda\cup\{M_x\}}}=M_x$; assume now that $M_y\in\Delta$. Then, $M_y$ is $\star$-bigger than $M_x$ and than every $I\in\buoni_y\setminus\{M_y\}$, and thus $M_y$ is not $\star_I$-closed for every $I\in\Lambda\cup\{M_x\}$. Since $M_y$ is an atom, it follows that $M_y$ is not $\star_{\Lambda\cup\{M_x\}}$-closed, while it is $\star_\Delta$-closed. Therefore, $\star_{\Lambda\cup\{M_x\}}\neq\star_\Delta$.

To show \ref{prop:xyconsecutivi:cupvscup} we can proceed like in the proof of Proposition \ref{prop:buoni-antichain}\ref{prop:buoni-antichain:d}, using the fact that $y\in I^{\star_{M_x}}$ for every $I\in\buoni_y$.
\end{proof}

To apply more clearly Propositions \ref{prop:buonia-buonib} and \ref{prop:xyconsecutivi}, we introduce the following notation. For each star operation $\star$, let $\qm(\star)$ be the biggest integer $x$ such that there is an $I\in\buoni_x$ such that $I$ is $\star$-closed; if $x$ does not exist, set $\qm(\star):=0$. Moreover, for an integer $x$, let $\numbuoni{x}(S)$ be the set of star operations such that $\qm(\star)=x$. The following lemma points out the main properties of $\qm$.
\begin{lemma}\label{lemma:qm}
Let $S$ be a numerical semigroup.
\begin{enumerate}[(a)]
\item If $\star\in\insstar(S)$, then either $\qm(\star)=0$ or $\qm(\star)\in\insN\setminus S$.
\item\label{lemma:qm:deltacupbuoni} If $\Delta\subseteq\bigcup_{x\in X}\buoni_x$ for some set $X$, then $\qm(\star_\Delta)=\max\{x: \buoni_x\cap\Delta\neq\emptyset\}$.
\item\label{lemma:qm:deltabuonix} If $\Delta\subseteq\buoni_x$ and $\Delta\neq\emptyset$, then $\qm(\star_\Delta)=x$.
\item\label{lemma:qm:v} $\qm(v)=0$.
\end{enumerate}
\end{lemma}
\begin{proof}
(a) If $x:=\qm(\star)\neq 0$, then there is an $I\in\buoni_x$ such that $I=I^\star$; however, $\buoni_x$ is nonempty if and only if $M_x$ is nondivisorial (Proposition \ref{prop:maxbuonia}) and in particular $x\in\insN\setminus S$.

(b) Let $y:=\max\{x: \buoni_x\cap\Delta\neq\emptyset\}$. If $I\in\Delta\cap\buoni_y$, then $I=I^\star$, so $\qm(\star_\Delta)\geq y$; on the other hand, if $J\in\buoni_z$ for some $z>y$, then $z\in J^{\star_\Delta}$, since $z\in J^v$ (by definition of $\buoni_z$) and $z\in(-\alpha+I)$ for any $I\in\buoni_x$ with $x<z$ and every $\alpha\geq 0$.

(c) follows directly from the previous point. For (d) it is enough to note that, if $I\in\buoni_x$, then by definition $I\neq I^v$.	
\end{proof}

To simplify the statement of the next corollary, we say that a nonempty subset $\Lambda\subseteq\nondiv(S)$ is \emph{good} if one of the following two conditions holds:
\begin{enumerate}
\item $\Lambda$ is an antichain, with respect to inclusion, of $\buoni_y$ (for some $y\in\insN\setminus S$);
\item $\Lambda=\Delta\cup\{M_x\}$, where $\Delta$ is a nonempty antichain of $\buoni_y\setminus\{M_y\}$ with respect to inclusion, and $x,y$ are as in Proposition \ref{prop:xyconsecutivi}.
\end{enumerate}

\begin{cor}\label{cor:xycons}
Let $S$ be a numerical semigroup, and let $\Lambda_1,\Lambda_2\subseteq\nondiv(S)$ be two good sets. If $\star_{\Lambda_1}=\star_{\Lambda_2}$, then $\Lambda_1=\Lambda_2$.
\end{cor}
\begin{proof}
The case in which both $\Lambda_i$ are antichains of some $\buoni_{y_i}$ is Proposition \ref{prop:buonia-buonib}.

Suppose $\Lambda_1=\Delta_1\cup\{M_x\}$, with $\Delta_1\subseteq\buoni_y\setminus\{M_y\}$; then, by Lemma \ref{lemma:qm}\ref{lemma:qm:deltacupbuoni}, $\qm(\star_{\Delta_1\cup\{M_x\}})=\sup\{x,y\}=y$. Since $\star_{\Lambda_1}=\star_{\Lambda_2}$, it must be $\qm(\star_{\Lambda_2})=y$; since $\Lambda_2$ is good, still by Lemma \ref{lemma:qm}, either $\Lambda_2\subseteq\buoni_y$ or $\Lambda_2=\Delta_2\cup\{M_x\}$ for some antichain $\Delta_2$ of $\buoni_y\setminus\{M_y\}$. By Proposition \ref{prop:xyconsecutivi}, the former case is impossible, while the latter implies $\Delta_2=\Delta_1$, i.e., $\Lambda_2=\Lambda_1$. The claim is proved.
\end{proof}

Corollary \ref{cor:xycons} can not be further extended to cover the case of the antichains $\Delta$ that are composed of arbitrary ideals in different $\buoni_a$. Indeed, let $S:=\langle 5,6,7,8,9\rangle=\{0,5,\rightarrow\}$. For every $I\in\nondiv(S)$, $I^v=\insN$, and thus $\nondiv(S)=\buoni_4\cup\buoni_3\cup\buoni_2\cup\buoni_1$. However, $S\cup\{4\}$ is not an atom (Corollary \ref{cor:g<2mu}) and so, by Proposition \ref{prop:ufd}, there are antichains $\Delta\neq\Lambda$ such that $\star_\Delta=\star_\Lambda$.

\begin{prop}\label{prop:stime-numbuoni}
Let $S$ be a numerical semigroup and let $T(S)=\{\tau_1<\cdots<\tau_t\}$; let $x,y,a\in\insN\setminus S$.
\begin{enumerate}[(a)]
\item\label{prop:stime-numbuoni:2x-3} If $x<y$ and $M_x$ is not divisorial, then $|\numbuoni{y}(S)|\geq 2\numantiincl(\buoni_y)-3$.
\item\label{prop:stime-numbuoni:tau} If $i\neq 1,t$, then $|\numbuoni{\tau_i}(S)|\geq 2\dedekind(i-1)-3$.
\item\label{prop:stime-numbuoni:g} $|\numbuoni{g}(S)|\geq 2\dedekind(t-1)-5$.
\item\label{prop:stime-numbuoni:g-a} If $\mu<a<g$ and $g-a\notin S$, then $\numantiincl(\buoni_a)\geq\dedekind(t-1)$.
\item\label{prop:stime-numbuoni:0} $|\numbuoni{0}(S)|\geq 1$.
\end{enumerate}
\end{prop}
\begin{proof}
\ref{prop:stime-numbuoni:2x-3} The existence of $x$ implies the existence of an $x'\in\insN\setminus S$ such that $x'<y$ and all integers between $x'$ and $y$ are in $S$. We have $\numantiincl(\buoni_y)-1$ nonempty antichains (with respect to inclusion) of $\buoni_y$, each of which induces a different star operation; by Proposition \ref{prop:xyconsecutivi} and Corollary \ref{cor:xycons}, if $\Lambda\neq\{M_y\}$ is one of these, then $\Lambda\cup\{M_{x'}\}$ gives a new star operation $\star$ with $\qm(\star)=y$, so we can add other $\numantiincl(\buoni_y)-2$ star operations.

\ref{prop:stime-numbuoni:tau} Consider the ideals of the form $S\cup \{x\inN: x>\tau_i\}\cup A$, for $A\subseteq\{\tau_1,\ldots,\tau_{i-1}\}$. Since $\tau_i\neq g$, all these are strictly bigger than $S$ and so are not divisorial, and they are in $\buoni_{\tau_i}$; therefore, by Proposition \ref{prop:buoni-antichain}, $\numantiincl(\buoni_{\tau_i})\geq\dedekind(i-1)$. By part \ref{prop:stime-numbuoni:2x-3}, $|\numbuoni{\tau_i}(S)|\geq 2\dedekind(i-1)-3$.

\ref{prop:stime-numbuoni:g} We can use the same proof of the previous point, only noting that the antichain composed of $A=\emptyset$ generates the $v$-operation, which is not in $\numbuoni{g}(S)$ but rather in $\numbuoni{0}(S)$. In the same way, $\{\emptyset\}\cup\{M_{x'}\}$ generates a star operation in $\numbuoni{x'}(S)$ rather than a star operation in $\numbuoni{g}(S)$.

\ref{prop:stime-numbuoni:g-a} Suppose $\mu<a$. Let $i$ be such that $\tau_{i-1}<a\leq\tau_i$ (with $\tau_0:=0$). If $j<i$, define $\eta_j:=\tau_j$. If $j>i$, define $\eta_j:=\tau_j-k_j\mu$, where $k_j\inN$ is such that $a-\mu<\tau_j-k_j\mu<a$. For every $A\subseteq\{\eta_1,\ldots,\eta_t\}$, the set $I_A:=A\cup S\cup\{x\inN: x>a\}$ is an ideal, $I_A\in\buoni_a$ and $I_A\subseteq I_B$ if and only if $A\subseteq B$; therefore, $\numantiincl(\buoni_a)\geq\dedekind(t-1)$.

\ref{prop:stime-numbuoni:0} follows from the fact that $v\in\numbuoni{0}(S)$.
\end{proof}

\begin{cor}\label{cor:sumdedekindt}
Let $S$ be a numerical semigroup. Then,
\begin{equation}\label{eq:sumdedekindt}
|\insstar(S)|\geq 2\left[\sum_{i=1}^{t-1}\dedekind(i)\right]-3(t-1).
\end{equation}
\end{cor}
\begin{proof}
If $t=2$, then the right hand side of \eqref{eq:sumdedekindt} is equal to $2\dedekind(1)-3=3$; since $S$ admits the three (different) star operations $v$, $\star_{M_g}$ and $\star_{M_\tau}$, the inequality is proved.

Suppose $t>2$ and let $T(S):=\{\tau_1,\ldots,\tau_t=g\}$. If $1<i<t$, then by the previous proposition we have $|\numbuoni{\tau_i}(S)|\geq 2\dedekind(i-1)-3$, while $|\numbuoni{\tau_t}(S)|\geq 2\dedekind(t-1)-5$. Moreover, $\numbuoni{\tau_1}(S)$ and $\numbuoni{0}(S)$ are nonempty, so that
\begin{equation*}
|\insstar(S)|\geq\sum_x|\numbuoni{x}(S)|\geq|\numbuoni{0}(S)|+|\numbuoni{\tau_1}(S)|+|\numbuoni{g}(S)|+\sum_{i=2}^{t-1}|\numbuoni{\tau_i}(S)|\geq
\end{equation*}
\begin{equation*}
\geq 2+2\dedekind(t-1)-5+\sum_{i=2}^{t-1}(2\dedekind(i-1)-3)=2\dedekind(t-1)-3+ \sum_{i=1}^{t-2}(2\dedekind(i)-3).
\end{equation*}
After a rearrangement, we obtain our claim.
\end{proof}

The proof above shows that the previous corollary does not give an useful estimate in the case $t=2$. However, yet when $t=3$ we get
\begin{equation*}
|\insstar(S)|\geq 2(\dedekind(2)+\dedekind(1))-3\cdot 2=2(6+3)-6=12,
\end{equation*}
and when $t=4$ we already have $|\insstar(S)|\geq 49$.

\begin{cor}\label{cor:tdedekindt}
Let $S$ be a numerical semigroup, and let $t:=t(S)$. If $\tau>\mu$ for every $\tau\in T(S)$, then 
\begin{equation*}
|\insstar(S)|\geq (2t-1)\cdot\dedekind(t-1)-3t+1.
\end{equation*}
\end{cor}
\begin{proof}
Let $T(S):=\{\tau_1,\ldots,\tau_t=g\}$, with $\tau_1$ being the smallest element. As in the proof of Corollary \ref{cor:sumdedekindt}, we have
\begin{equation*}
|\insstar(S)|\geq|\numbuoni{0}(S)|+\sum_{i=1}^t|\numbuoni{\tau_i}(S)|.
\end{equation*}
Clearly, $|\numbuoni{0}(S)|\geq 1$, while $|\numbuoni{g}(S)|\geq 2\dedekind(t-1)-5$ by Proposition \ref{prop:stime-numbuoni}\ref{prop:stime-numbuoni:g}. If $i\neq t$, then by Proposition \ref{prop:stime-numbuoni}\ref{prop:stime-numbuoni:g-a} we have $\numantiincl(\buoni_{\tau_i})\geq\dedekind(t-1)$; hence, $|\numbuoni{\tau_1}(S)|\geq\dedekind(t-1)-1$ by Proposition \ref{prop:buoni-antichain}\ref{prop:buoni-antichain:d}. On the other hand, if $i\neq 1$, then Proposition \ref{prop:stime-numbuoni}\ref{prop:stime-numbuoni:2x-3} implies that $|\numbuoni{\tau_i}(S)|\geq 2\numantiincl(\buoni_{\tau_i})-3\geq 2\dedekind(t-1)-3$. Therefore, 
\begin{equation*}
|\insstar(S)|\geq 1+[\dedekind(t-1)-1]+[2\dedekind(t-1)-5]+(t-2)[2\dedekind(t-1)-3]=\eql
=(1+2+2t-4)\dedekind(t-1)-5-3t+6=(2t-1)\dedekind(t-1)-3t+1.
\end{equation*}
The claim is proved.
\end{proof}

The estimates in $t$, despite being useful, are not quite enough to restrict the range of possible semigroups with a low number of star operations; we would like instead to have estimates that depend on $\mu$ or on $g$. The following propositions, analyzing different cases, tackle this problems, mirroring and strenghtening \cite[Propositions 4.11-4.14]{semigruppi_main}. In the following, we will not give any direct estimate on the size of $\insstar(S)$, since they can be obtained patching together various results. However, we will use the bounds we obtain here in Section \ref{sect:explicit}, where we will determine the semigroups with a small number of star operations.

\begin{prop}\label{prop:dednu}
Let $S$ be a numerical semigroup, and let $\nu:=\left\lceil\frac{\mu-1}{2}\right\rceil$; let $a\leq g/2$ be a positive integer such that $a,g-a\notin S$.
\begin{enumerate}[(a)]
\item\label{prop:dednu:mu} If $a>\mu$, then $\numantiincl(\buoni_a)\geq\dedekind(\nu)$.
\item\label{prop:dednu:2mu} If $a>2\mu$, then $\numantiincl(\buoni_a)\geq 2\dedekind(\nu)-2$.
\end{enumerate}
\end{prop}
\begin{proof}
Let $X:=\{x_1,\ldots,x_\eta\}$ be the set of integers not belonging to $S$ and comprised between $a-\mu$ and $a$ (extremes excluded). By \cite[Lemma 4.13]{semigruppi_main}, $|X|\geq\nu$.

\ref{prop:dednu:mu} Each set $A\subseteq X$ generates an ideal $S\cup\{x\inN: x>a\}\cup A$, and all of these are in $\buoni_a$ (since $g-a\notin S$). Thus, the number of antichains in $\buoni_a$, with respect to inclusion, is at least $\dedekind(\eta)\geq\dedekind(\nu)$.

\ref{prop:dednu:2mu} For every $x_i\in X$, $x_i>\mu$, since $a>2\mu$. Let $y_i:=a-x_i$; then, $y_i<\mu$, so that $y_i\notin S$ and $X\cap Y=\emptyset$. Let $Y:=\{y_1,\ldots,y_\eta\}$ and let $I:=S\cup\{x\inN:x>a\}$. For each $A\subseteq X$ (respectively, $A\subseteq Y$), $I_A:=I\cup(A+S)$ is an ideal which does not contain $a$, and thus $I_A\in\buoni_a$; moreover, $I_A\cap X=A$ (resp., $I_A\cap Y=A$), so that if $I_A\subseteq I_B$ then $A\subseteq B$.

Therefore, each antichain of the power set of $X$, and each antichain of the power set of $Y$ (both with respect to inclusion), give rise to an antichain of $\buoni_a$ (with respect to inclusion). Moreover, the empty antichain and the antichain composed of the empty set belong to both power sets, while all the others are different; therefore, $\numantiincl(\buoni_a)\geq 2\dedekind(\eta)-2\geq 2\dedekind(\nu)-2$.
\end{proof}

If $a\inN\setminus S$ is smaller than $\mu$, we have to adopt a slightly different method.
\begin{prop}\label{prop:dedekind-a<mu}
Let $S$ be a numerical semigroup and $a$ be a positive integer such that $a<\mu$ and $g-a\notin S$. Then:
\begin{enumerate}[(a)]
\item\label{prop:dedekind-a<mu:a} $\numantiincl(\buoni_a)\geq\dedekind(a-1)$;
\item\label{prop:dedekind-a<mu:b} if $a<s<\mu$, then $\numantiincl(\buoni_s)\geq\dedekind(s-2)$.
\end{enumerate}
\end{prop}
\begin{proof}
\ref{prop:dedekind-a<mu:a} Define $I:=\{0\}\cup\{x\inN: x>a\}$. For each subset $A\subseteq\{1,\ldots,a-1\}$, $I\cup A$ is a nondivisorial ideal of $S$, and it belongs to $\buoni_a$. Hence, $\buoni_a$ has at least $\dedekind(a-1)$ antichains (with respect to ordering). 

\ref{prop:dedekind-a<mu:b} Let $s\inN$ such that $a<s<\mu$, and define $A_s:=\{1,\ldots,s-1\}\setminus\{s-a\}$ and $I_s:=S\cup\{x\inN:x>s\}$. We claim that, for every $B\subseteq A_s$, the ideal $J:=I_s\cup B\cup\{s-a\}$ belongs to $\buoni_s$.

Indeed, suppose $s\notin J^v$. Then, there is a $\gamma\inN$ such that $J\subseteq-\gamma+S$ but $s\notin-\gamma+S$. In particular, since $s=\sup(\insN\setminus J)$, it must be $\gamma=g-s$; thus, $-\gamma+(g-a)=s-a\notin-\gamma+S$. However, this would imply $J\nsubseteq-\gamma+S$, against the hypothesis. Therefore, $J\in\buoni_s$.

It now follows from Proposition \ref{prop:buoni-antichain} that $\numantiincl(\buoni_s)\geq\dedekind(s-2)$.
\end{proof}

We end this section by using the methods we developed to calculate the number of star operations in one particular case.

\begin{ex}\label{es:4567}
The star operations of $S:=\langle 4,5,6,7\rangle=\{0,4,\rightarrow\}$.

The ideals of $\insfracid_0(S)$ are in the form $S\cup A$, where $A\subseteq\{1,2,3\}$, and every such $A$ is acceptable. Moreover, $S\cup A$ is divisorial if and only if $A=\emptyset$ or $A=\{1,2,3\}$. To ease the notation, we set $I(a):=S\cup\{a\}$ and $I(a,b):=S\cup\{a,b\}$.

Since $I^v=\insN$ if $I\in\insfracid_0(S)$ and $I$ is not divisorial, every ideal of $\nondiv(S)$ belongs to $\buoni_a$, for some $a$: to be specific,
\begin{itemize}
\item $\buoni_3=\{I(1,2),I(1),I(2)\}$;
\item $\buoni_2=\{I(1,3),I(3)\}$;
\item $\buoni_1=\{I(2,3)\}$.
\end{itemize}

Since $M_a=\insN\setminus\{a\}$, we have $I(1,2)=M_3$, $I(1,3)=M_2$ and $I(2,3)=M_1$. Hence, $I(1,2)$ is the maximum of $\nondiv(S)$ and $I(1,2)\geq_\star I(1,3)\geq_\star I(2,3)$. Since $I(3)=I(2,3)\cap I(1,3)$, we also have $I(1,3)\leq_\star I(3)$. If $I$ is equal either to $I(2,3)$ or to $I(3)$, and $0\in -a+I$, then either $a=0$ or $\insN\subseteq-a+I$; therefore, $I(2,3)$ and $I(3)$ are minimal elements of $(\nondiv,\leq_\star)$.

By Proposition \ref{prop:buoni-antichain}, $I(1)$ and $I(2)$ are not $\star$-comparable. If $(-a+I(1))\cap\insN\in\nondiv(S)$, then $a$ is equal either to 0 or to 1; therefore $I(3)\leq_\star I(1)$, and since $I(1)\cap I(3)=S$ there are no other $\star_{I(1)}$-closed ideals. In the same way, the unique $\star_{I(2)}$-closed ideals in $\nondiv(S)$ are $I(2)$ and $I(2,3)$. The last ideal to be considered is $I(1,3)$. By the proof of Proposition \ref{prop:buonia-buonib}, $I(1,3)$ is not $\star$-bigger than $I(1)$ and $I(2)$ and, by the above reasoning, nor is $\star$-minor than them. In conclusion, we get the Hasse diagram of $(\nondiv(S),\leq_\star)$, which is pictured in Figure \ref{fig:4567}.

\begin{figure}
\begin{equation*}
\begin{tikzcd}
& \arrow{dl}I(1,2)\arrow{d}\arrow{dr}\\
I(2)\arrow{dr} & I(1,3)\arrow{d}\arrow{dr} & I(1)\arrow{d}\\
& I(2,3) & I(3)
\end{tikzcd}
\end{equation*}
\caption{Hasse diagram of $\nondiv(\langle 4,5,6,7\rangle)$.}
\label{fig:4567}
\end{figure}

Every $I(a)$ is in $\buoni_b$, for some $b$, and $|M_b\setminus I(a)|=1$; therefore, applying Proposition \ref{prop:atom-MaI}, every principal star operation is prime, and by Proposition \ref{prop:ufd} the number of star operations on $S$ is equal to the number of antichains of $(\nondiv(S),\leq_\star)$. Counting, we see that $\nondiv(S)$ contains 7 antichains with two or more elements: adding 6 principal star operations and the empty antichain (corresponding to the $v$-operation), we get $|\insstar(S)|=14$.
\end{ex}

\section{The pseudosymmetric case}\label{sect:pseudosymm}
A semigroup $S$ is called \emph{pseudosymmetric} if $g(S)$ is even and $T(S)=\{g,g/2\}$ or, equivalently, if $g(S)$ is even and $g-a\in S$ for every $a\in\insN\setminus S$, $a\neq g/2$.

\begin{prop}\label{prop:pseudosymm-pozzo}
Let $S$ be a pseudosymmetric semigroup. The unique minimal element of $\nondiv(S)$ is $S\cup\{g\}$.
\end{prop}
\begin{proof}
Let $I:=S\cup\{g\}$ and let $\tau:=g/2$. It is enough to show that $I$ is $\star_J$-closed for each nondivisorial ideal $J\in\insfracid_0(S)$. If $g\notin J$, then $J=S\cup\{\tau\}=M_g$ is the maximum of $(\nondiv,\leq_\star)$.

Suppose $g\in J$. If $\tau\notin J$, then $I=J\cap(S-M)$; since $(S-M)$ is divisorial, $I$ is $\star_J$-closed. Suppose $\tau\in J$ and consider the ideal $L:=(J-(J-I))$; note that it contains $g$ since $J$ contains all the integers greater or equal than $g$. If $\tau\notin L$, then $I=L\cap(S-M)$ is $\star_J$-closed. Otherwise, $\tau+(J-I)\subseteq J$. However,
\begin{equation*}
(J-I)=(J-(S\cup\{g\}))=(J-S)\cap(J-g)=J
\end{equation*}
(the last equality coming from $(J-S)=J$ and $g\in J$); therefore, $\tau+J\subseteq J$. By \cite[Proposition I.1.16]{fontana_maximality}, this would imply that $J$ is divisorial, against our assumption. Therefore, $I$ must be $\star_J$-closed.
\end{proof}

\begin{prop}\label{prop:pseudosymm-Itau}
Let $S$ be a pseudosymmetric semigroup, and let $\tau:=g/2$. Then:
\begin{enumerate}[(a)]
\item\label{prop:pseudosymm-Itau:a} if $I\in\insfracid_0(S)$, $I\neq S$ and $\tau\notin I$, then $I^v=I\cup\{\tau\}$;
\item\label{prop:pseudosymm-Itau:ordine} if $I,J\in\buoni_\tau$, then $I\geq_\star J$ if and only if $I\supseteq J$.
\end{enumerate}
\end{prop}
\begin{proof}
\ref{prop:pseudosymm-Itau:a} By \cite[Proposition I.1.16]{fontana_maximality}, and since $\tau\in T(S)$ (so that $I\neq I^v$ by \cite[Proposition 3.11]{semigruppi_main}), it is enough to show that $\tau+(I\cup\{\tau\})\subseteq(I\cup\{\tau\})$. However,
\begin{equation*}
\tau+(I\cup\{\tau\})=\tau+(\{0\}\cup M\cup(I\setminus S)\cup\{\tau\})=\eqlnn
=\{\tau,g\}\cup(\tau+M)\cup(\tau+(I\setminus S)).
\end{equation*}

The first two sets are contained in $I\cup\{\tau\}$ because $\tau\in(S-M)$. If now $x\in I\setminus S$, then either $x>\tau$ (and so $x+\tau>g$ and $x+\tau\in S$) or $x<\tau$, and so $\tau-x\notin S$ (otherwise $\tau\in I$); in the latter case, $g-(\tau-x)\in S$, but $g-(\tau-x)=\tau+x$, and thus $x+\tau\in S\subseteq I$.

\ref{prop:pseudosymm-Itau:ordine} If $I\geq_\star J$, then $I\supseteq J$ by Proposition \ref{prop:buoni-antichain}. Suppose $J\subseteq I$. Then,
\begin{equation*}
J^{\star_I}\subseteq J^v\cap I=(J\cup\{\tau\})\cap I=J
\end{equation*}
since $\tau\notin I$. Hence, $\star_J\geq\star_I$.
\end{proof}

A direct consequence of this proposition is a direct formula for the number of star operations in a particular class of semigroups.

\begin{prop}\label{prop:pseudosymm-2g-2}
Let $S:=\{0,\mu,\mu+1,\ldots,2\mu-3,2\mu-1,\rightarrow\}$, where $\mu\geq 3$. Then, $|\insstar(S)|=1+\dedekind(\mu-2)$.
\end{prop}
\begin{proof}
It is clear that $g:=g(S)=2\mu-2$. Let $\tau:=g/2=\mu-1$; then, $T(S)=\{g,\tau\}$, so that $S$ is pseudosymmetric.

If $I\in\insfracid_0(S)$ is an ideal not containing $g$, then $I$ is either $S$ or $S\cup\{\tau\}$. Moreover, if $(S-M)\subseteq I$, then every element greater than $\tau$ is in $I$ and thus $\tau+I\subseteq I$, and it follows from \cite[Proposition I.1.16]{fontana_maximality} that any such $I$ is divisorial. By Proposition \ref{prop:pseudosymm-Itau}, if $I$ contains $g$ but not $\tau$, then $I^v=I\cup\{\tau\}$. Define $I_A:=S\cup A\cup\{g\}$. Then,
\begin{equation*}
\nondiv=\{S\cup\{\tau\}\}\cup\{I_A: A\subseteq\{1,\ldots,\mu-2\}\}.
\end{equation*}
By Propositions \ref{prop:atom-MaI} and \ref{prop:anticatene-V1} every ideal is thus an atom; by Proposition \ref{prop:ufd}, $|\insstar(S)|=\numanti(\nondiv)$.

The ideal $M_g=S\cup\{\tau\}$ generates the identity. Moreover, each $I_A$ is in $\buoni_\tau$; by Proposition \ref{prop:pseudosymm-Itau}\ref{prop:pseudosymm-Itau:ordine},  $\star_{I_A}\geq\star_{I_B}$ if and only if $I_A\supseteq I_B$, i.e., if and only if $A\supseteq B$.

Therefore, if $\Delta$ is an antichain of $\nondiv$, then either $\Delta=\{M_g\}$ or $\Delta$ is an antichain of $\mathcal{P}(\{1,\ldots,\mu-2\})$. Hence $|\insstar(S)|=\dedekind(\nondiv)=1+\dedekind(\mathcal{P}\{1,\ldots,\mu-2\})=1+\dedekind(\mu-2)$.
\end{proof}

\section{Explicit calculation}\label{sect:explicit}
In this section, we shall use the estimates we built in the previous sections to determine explicitly all the numerical semigroups $S$ such that $2\leq|\insstar(S)|\leq 10$.

\caso $\mu(S)=3$.

We shall use the following.
\begin{teor}[\protect{\cite[Theorem 7.6]{semigruppi_mu3}}]\label{teor:main-mu3}
Let $S=\langle 3,3\alpha+1,3\beta+2\rangle$ be a numerical semigroup. Then, $|\insstar(S)|=\binom{\alpha+\beta+1}{2\alpha-\beta}$.
\end{teor}

Equivalently, numerical semigroups of multiplicity 3 with exactly $n$ star operations are in bijective correspondence with binomial coefficients $\binom{a}{b}$ such that $\binom{a}{b}=n$ and $a+b\equiv 1\bmod 3$ (see \cite[Proposition 8.2]{semigruppi_mu3}).

Suppose $x:=\binom{a}{b}$ is a binomial coefficient such that $x\leq 10$. Then, $a\leq 10$; the unique possibilities with $a+b\equiv 1\bmod 3$ are the following.

\begin{itemize}[itemsep=1ex]
\item $\binom{\alpha+\beta+1}{2\alpha-\beta}=\binom{3}{1}=3$: then, $\alpha=1$ and $\beta=1$, so $S=\langle 3,4,5\rangle$.
\item $\binom{\alpha+\beta+1}{2\alpha-\beta}=\binom{4}{3}=4$: then, $\alpha=2$ and $\beta=1$, so $S=\langle 3,5,7\rangle$.
\item $\binom{\alpha+\beta+1}{2\alpha-\beta}=\binom{5}{2}=10$: then, $\alpha=2$ and $\beta=2$, so $S=\langle 3,7,8\rangle$.
\item $\binom{\alpha+\beta+1}{2\alpha-\beta}=\binom{6}{1}=6$: then, $\alpha=2$ and $\beta=3$, so $S=\langle 3,7,11\rangle$.
\item $\binom{\alpha+\beta+1}{2\alpha-\beta}=\binom{7}{6}=7$: then, $\alpha=4$ and $\beta=2$, so $S=\langle 3,8,13\rangle$.
\item $\binom{\alpha+\beta+1}{2\alpha-\beta}=\binom{9}{1}=9$: then, $\alpha=3$ and $\beta=5$, so $S=\langle 3,10,17\rangle$.
\item $\binom{\alpha+\beta+1}{2\alpha-\beta}=\binom{10}{9}=10$: then, $\alpha=6$ and $\beta=3$, so $S=\langle 3,11,19\rangle$.
\end{itemize}

\bigskip

Suppose now $\mu>3$. If $|\insstar(S)|\geq 2$, $S$ is not symmetric; therefore, we can suppose $t(S)>1$, and thus there is an $\tau$ such that $\tau,g-\tau\notin S$ and $0<\tau\leq g/2$.

\caso $\tau\neq g/2$ and $\mu>3$.

Let $\lambda:=g-\tau$; by hypothesis, $g>\lambda>g/2$, and in particular $\tau\neq \lambda$. The set $\buoni_\lambda$ contains at least two elements: $M_\lambda$ and $I_\lambda:=S\cup\{x\inN: x>\lambda\}$ (which is indeed different from $M_\lambda$: if $\lambda>\mu$ then $g-k\mu\in M_\lambda\setminus I_\lambda$ for some $k$, while if $\lambda<\mu<$, since $1<\lambda$, we have $\lambda-1\in M_\lambda\setminus I_\lambda$). Hence, $\numantiincl(\buoni_\lambda)\geq 3$ and, by Proposition \ref{prop:stime-numbuoni}\ref{prop:stime-numbuoni:2x-3}, $|\numbuoni{\lambda}(S)|\geq 3$. Moreover, also $\buoni_g$ contains at least two elements ($S\cup(\tau+S)$ and $S\cup(\lambda+S)$) and thus $|\numbuoni{g}(S)|\geq 3$. Adding the $v$-operation we get at least 7 star operations.

If $\tau>\mu$, then by Proposition \ref{prop:dednu} $\numantiincl(\buoni_\tau)\geq\dedekind(\nu)=\dedekind(2)=6$, so we get $5=6-1$ new star operations; suppose $\tau<\mu$. If $\tau=\mu-1$, then we have (by Proposition \ref{prop:dedekind-a<mu}) $\numantiincl(\buoni_\tau)\geq\dedekind(\mu-2)\geq\dedekind(2)$ and again 5 new star operations; if $\tau<\mu-1$, then $\numantiincl(\buoni_{\mu-1})\geq\dedekind(\mu-3)$ and thus (again by Proposition \ref{prop:stime-numbuoni}\ref{prop:stime-numbuoni:2x-3}) we have $|\numbuoni{\mu-1}(S)|\geq 2\dedekind(1)-3=3$ new star operations, for a total of 10. To them we must add $\star_{M_\tau}$, putting the total to 11.

Therefore, no semigroups arise from this case.

\caso $\tau=g/2$ and $\mu>3$.

We can suppose that no other couple $\{b,g-b\}$ is out of $S$, for otherwise we fall in the previous case; therefore, $S$ must be pseudosymmetric.

By Propositions \ref{prop:dednu} and \ref{prop:dedekind-a<mu}, $|\insstar(S)|$ is bigger than at least one between $\dedekind(\nu)-1$ and $\dedekind(\mu-3)-1$ (where $\nu:=\left\lceil\frac{\mu-1}{2}\right\rceil$); if $\mu\geq 6$, then both $\nu$ and $\mu-3$ are at least 3, and thus $|\insstar(S)|\geq\dedekind(3)-1=19$. Hence, we can suppose $\mu$ equals to 4 or 5.

If $\tau<\mu-1$, then $g<2\mu-2$, and thus $\mu-1\in T(S)$. But this contradicts the pseudosymmetricity of $S$.

If $\tau=\mu-1$, then we can apply Proposition \ref{prop:pseudosymm-2g-2} to obtain $|\insstar(S)|=1+\dedekind(\mu-2)$. If $\mu=4$ we have $|\insstar(S)|=1+\dedekind(2)=7$, while if $\mu=5$ we have $|\insstar(S)|=1+\dedekind(3)=21$.

If $\tau>2\mu$, then by Proposition \ref{prop:dednu} $\numantiincl(\buoni_\tau)\geq 2\dedekind(\nu)-2\geq 2\cdot 6-2=10$; hence, we get 9 star operations, which becomes 11 if we count $d=\star_{M_g}$ and the $v$-operation. Therefore, $\tau<2\mu$.

Thus, we need to consider the case $\mu+1\leq\tau\leq 2\mu-1$. If $\tau=2\mu-1$ then the same proof of Proposition \ref{prop:dednu}\ref{prop:dednu:2mu} shows that $\numantiincl(\buoni_\tau)\geq 2\dedekind(\nu)-2\geq 2\cdot 6-2=10$, and as before $|\insstar(S)|\geq 11$. 

Suppose $\mu=5$. Let $X:=\{b\inN\setminus S:\tau-\mu<b<\tau\}$ and $Y:=\{b\inN\setminus S:\tau<b<\tau+\mu\}$; we have $|X|\geq 2$, and since $S$ is pseudosymmetric $|X|+|Y|=\mu-1=4$. If $|X|=3$, then by the proof of Proposition \ref{prop:dednu} $|\numbuoni{\tau}(S)|\geq\dedekind(3)-1=19$ and $|\insstar(S)|>10$. Hence $|X|=|Y|=2$; let $Y=\{b,b'\}$, with $b<b'$. If $I_a:=S\cup\{x\inN: x>a\}$, then $I_\tau\cup A$ is a non-divisorial ideal for every $A\subseteq X$; moreover, $I_{b'}$, $I_{b'}\cup\{b\}$ and $M_{b'}$ are non-divisorial (and different because $\tau\in M_{b'}$), and likewise $I_b$ and $M_b$ are different. Adding also $M_g$ (note that $g>b'$ since $g-\tau=\tau>\mu$), we have 10 non-divisorial ideals and thus 11 star operations.

Suppose $\mu=4$; we have to check the cases $\tau=5$ and $\tau=6$. The latter case is impossibile since it would imply $g=2\tau=12$; hence, suppose $\tau=5$. An easy calculation shows that $S$ must be equal to $\langle 4,7,9\rangle$, and that $\insN\setminus S=\{1,2,3,5,6,10\}$. As before, $\buoni_5$ has 6 antichains, and induces 5 star operations; moreover, $\buoni_6$ has two elements ($S\cup\{1,5,10\}$ and $S\cup\{1,3,5,10\}$) and thus it generates 3 (different) star operations. Adding the identity (generated by $M_g$) and the $v$-operation we get 10 star operations. Finally, $I:=S\cup\{g\}=S\cup\{10\}$ is not in any $\buoni_x$ (since $I^v=I\cup\{5\}$), and by Proposition \ref{prop:pseudosymm-pozzo} it is a minimal element of $\nondiv(S)$; it follows that $\star_I\neq\star_\Delta$ for every $\Delta\subseteq\nondiv(S)$, $\Delta\neq\{I\}$. Hence we get also an eleventh star operation.

Therefore, the pseudosymmetric case yields we get the unique possibility $\mu=4$ and $\tau=\mu-1$, that is, $S=\langle 4,5,7\rangle$.

\bigskip

We have proved the following:
\begin{teor}\label{teor:main}
Let $S$ be a numerical semigroup which is not symmetric. Then, $|\insstar(S)|\leq 10$ if and only if one of the following holds:
\begin{enumerate}[(a)]
\item $S=\langle 3,4,5\rangle$, and $|\insstar(S)|=3$;
\item $S=\langle 3,5,7\rangle$, and $|\insstar(S)|=4$;
\item $S=\langle 3,7,11\rangle$, and $|\insstar(S)|=6$;
\item $S=\langle 3,8,13\rangle$, and $|\insstar(S)|=7$;
\item $S=\langle 4,5,7\rangle$, and $|\insstar(S)|=7$;
\item $S=\langle 3,10,17\rangle$, and $|\insstar(S)|=9$;
\item $S=\langle 3,7,8\rangle$, and $|\insstar(S)|=10$;
\item $S=\langle 3,11,19\rangle$, and $|\insstar(S)|=10$.
\end{enumerate}
\end{teor}

\section{Estimates}
The work done in Section \ref{sect:explicit} can, in principle, be replicated to find explicitly, given an arbitrary $n$, the number of numerical semigroups whose number of star operations is comprised between 2 and $n$. However, the efficiency of this enterprise decreases with the increase of $n$, partly due to the increase of the number of the different cases we have to consider, and partly due to the fact that we must consider more and more different special cases, each one requiring a different way to find ``good'' estimates. In this section, we use a different point of view, concentrating on finding an asymptotic estimate on the number of semigroups with $n$ or less star operations.

Let $\xi(n)$ denote the number of numerical semigroups with exactly $n$ star operations: by \cite[Theorem 4.15]{semigruppi_main}, $\xi(n)<\infty$ for every $n>1$. Denote also by $\Xi(n)$ the number of numerical semigroups $S$ with $2\leq|\insstar(S)|\leq n$; i.e., $\Xi(n)=\sum_{i=2}^n\xi(i)$. Recall also that, given two functions $f$ and $g$, the notation $f(n)=O(g(n))$ means that $\limsup_{n\to\infty}\frac{f(n)}{g(n)}<\infty$.

We start with an improvement of Propositions 4.11 and 4.12 of \cite{semigruppi_main}.
\begin{prop}\label{prop:insstar-delta}
Let $S$ be a non-symmetric semigroup. Then, $|\nondiv(S)|\geq\delta(S)$, and thus $|\insstar(S)|\geq\delta(S)+1$.
\end{prop}
\begin{proof}
Let $g:=g(S)$. Since $S$ is not symmetric, there is a $\tau\in T(S)\setminus\{g\}$; let $\lambda:=\min\{\tau,g-\tau\}$ (note that it may be $\tau=g-\tau$, and that both $\lambda$ and $g-\lambda$ are not in $S$). Consider the three sets
\begin{equation*}
A:=\{x\in \insN\setminus S: x<\lambda,\lambda-x\notin S\},\eql
B:=\{x\in \insN\setminus S: x<\lambda,\lambda-x\in S\},\eql
C:=\{x\in \insN\setminus S: x\geq\lambda\}.
\end{equation*}
Since $\insN\setminus S=A\cup B\cup C$, we have $\delta(S)=|A|+|B|+|C|$; we will define for every $x\in\insN\setminus S$ a different non-divisorial ideal $I_x$, whose definition depends on whether $x\in A$, $x\in B$ or $x\in C$.

If $x\in C$, then define $I_x:=M_x$; since $x\geq\lambda$ and $g-\lambda\notin S$, by Proposition \ref{prop:maxbuonia} $I_x\in\nondiv(S)$.

If $x\in A$, then $x\in M_\lambda$ \cite[Lemma 4.2]{semigruppi_main}; we define $I_x:=S\cup\{z\in\insN: z>x,z\in M_\lambda\}$; then, $\sup(\insN\setminus I_x)=\lambda$, and thus $I_x$ is not divisorial by \cite[Lemma 4.7]{semigruppi_main}. Moreover, $\sup(M_\lambda\setminus I_x)=x$, and thus $I_x\neq I_y$ if $x\neq y$ are in $A$.

If $x\in B$, consider $y:=g-\lambda+x$. Then, $y=g-(\lambda-x)$, and since $\lambda-x\in S$, we have $y\notin S$; moreover, $g-\lambda<y<g$. Let $I_x:=S\cup\{z\in\insN: z>y\}$; then, $g$ belongs to $I_x$ while $\tau$ does not, and thus $I_x$ is not divisorial. Moreover, $\sup(\insN\setminus I_x)=y$ (so that $I_x\neq I_w$ if $x\neq w$ are in $B$) and $M_y$ contains $g-\lambda$ (since $x\notin S$); hence, $I_x\neq M_y$.

It is straightforward to see that $I_x\neq I_y$ if $x$ and $y$ belong to different subsets; therefore, $\{I_x: x\in\insN\setminus S\}$ is a set of $\delta(S)$ non-divisorial ideals. In particular, $|\nondiv(S)|\geq\delta(S)$, and $|\insstar(S)|\geq\delta(S)+1$ (since we can also consider the $v$-operation).
\end{proof}

Let $d(n)$ be the number of numerical semigroups such that $\delta(S)=n$. It has been proved that there is a constant $C$ such that
\begin{equation*}
\lim_{n\to\infty}\frac{d(n)}{\phi^n}=C,
\end{equation*}
where $\phi=\frac{1+\sqrt{5}}{2}$ is the golden ratio \cite{fibonacci_like}; thus, there is a constant $D$ such that $d(n)\leq D\phi^n$. Hence,
\begin{equation*}
\Xi(n)\leq\sum_{i=1}^{n-1}d(i)\leq\sum_{i=1}^{n-1}D\phi^i\leq\frac{D\phi}{\phi-1}\phi^{n-1}=D'\phi^n.
\end{equation*}
Thus, Proposition \ref{prop:insstar-delta} implies $\Xi(n)=O(\phi^n)=O(e^{n\log\phi})$.

A more effective way to find estimates is to separate semigroups by multiplicity; that is, instead of working directly with $\Xi(n)$, we will use instead the functions $\Xi_\mu(n)$ that count the numerical semigroups $S$ with multiplicity $\mu$ and $2\leq|\insstar(S)|\leq n$. The two needed steps are, thus, to find a bound on $\Xi_\mu(n)$ and for the maximum admissible $\mu$. We start from the latter.
\begin{prop}\label{prop:muloglog}
For every $\epsilon>0$ there is an integer $n_0$ such that, for every $n\geq n_0$, if $S$ is a nonsymmetric numerical semigroup such that $|\insstar(S)|\leq n$, then
\begin{equation}\label{eq:stimamu}
\mu(S)\leq\left[\frac{2}{\log(2)}+\epsilon\right]\log\log(n).
\end{equation}
\end{prop}
\begin{proof}
Let $S$ be a nonsymmetric semigroup; then, there is a $x$ such that $x,g-x\notin S$: if $x<\mu$ we have $|\insstar(S)|\geq\dedekind(\mu-3)$ (by Proposition \ref{prop:dedekind-a<mu}), while if $x>\mu$, we have $|\insstar(S)|\geq\dedekind(\nu)$ (by Proposition \ref{prop:dednu}), where $\nu:=\left\lceil\frac{\mu-1}{2}\right\rceil$.

The quantity on the right hand side of \eqref{eq:stimamu} goes to infinity; therefore, for large $n$, we can restrict ourselves to $\mu(S)\geq 5$, so that $\nu\leq\mu-3$ and $|\insstar(S)|\geq\dedekind(\nu)$.

For any integer $k$, no two subset of $\{1,\ldots,k\}$ of cardinality $\lceil k/2\rceil$ are comparable; therefore, every family of such subsets is an antichain of $\mathcal{P}(\{1,\ldots,k\})$. Hence, 
\begin{equation*}
\log_2\dedekind(k)\geq\binom{k}{\lceil k/2\rceil}.
\end{equation*}
For large $a$, the binomial coefficient $\binom{2a}{a}$ is asymptotic to $\frac{2^{2a}}{\sqrt{\pi a}}$; in particular, for every $\epsilon_0$ and large enough $a$ (where ``large enough'' depends on $\epsilon_0$) we have $\binom{2a}{a}>2^{a(2-\epsilon_0)}$. Thus, for every $\epsilon_1$ there is a $\nu_0$ such that, if $\nu_1\geq\nu_0$, we have
\begin{equation*}
\log_2(\dedekind(\nu_1))\geq 2^{\frac{\nu_1}{2}(2-2\epsilon_1)}=2^{\nu_1(1-\epsilon_1)}.
\end{equation*}

Fix an $\epsilon$, and take an $\epsilon_1<\frac{\epsilon}{A+\epsilon}$, where $A:=\frac{2}{\log(2)}$; find $\nu_0$ as above, let $n'_0:=\dedekind(\nu_0)$, and take a $n\geq n'_0$. Moreover, choose the maximal $\overline{\mu}$ such that $n\geq\dedekind\left(\frac{\overline{\mu}-1}{2}\right)$, so that $\overline{\nu}:=\frac{\overline{\mu}-1}{2}\geq\nu_0$. For any semigroup $S$ such that $|\insstar(S)|\leq n$, we must have $\mu(S)\leq\overline{\mu}$ and $\nu(S)\leq\overline{\nu}$. Hence,
\begin{equation*}
\log_2(n)\geq\log_2(\dedekind(\overline{\nu}))\geq 2^{\overline{\nu}(1-\epsilon_1)},
\end{equation*}
i.e., $\log(n)\geq\log(2)\cdot 2^{\overline{\nu}(1-\epsilon_1)}$. Taking logarithms,
\begin{equation*}
\log\log(n)\geq\log(2)\cdot\left[\log_2(\log(2)\cdot 2^{\overline{\nu}(1-\epsilon_1)})\right]= \log\log(2)+\log(2)(\overline{\nu}(1-\epsilon_1)).
\end{equation*}
Isolating $\overline{\nu}$, we have
\begin{equation*}
\overline{\nu}\leq\inv{(1-\epsilon_1)\log(2)}(\log\log(n)-\log\log(2)),
\end{equation*}
and substituing $\overline{\nu}$ with $\frac{\overline{\mu}-1}{2}$ we have
\begin{equation*}
\mu(S)\leq\overline{\mu}\leq\frac{2}{\log(2)(1-\epsilon_1)}(\log\log(n)-\log\log(2))+1=\eql =\frac{A}{1-\epsilon_1}\log\log(n)+\left[1-\frac{A}{1-\epsilon_1}\log\log(2)\right].
\end{equation*}

The inequality $\epsilon_1<\frac{\epsilon}{A+\epsilon}$ implies that
\begin{equation*}
\epsilon>\epsilon_1(A+\epsilon)\Longrightarrow\epsilon>\frac{\epsilon_1A}{1-\epsilon_1};
\end{equation*}
therefore,
\begin{equation*}
A+\epsilon-\frac{A}{1-\epsilon}>A+\frac{\epsilon_1A}{1-\epsilon_1}-\frac{A}{1-\epsilon}= A+\frac{\epsilon_1-1}{1-\epsilon_1}A=0,
\end{equation*}
or equivalently $A+\epsilon>\frac{A}{1-\epsilon}$. Hence, there is a $n_0\geq n'_0$ such that, whenever $n\geq n_0$, we have
\begin{equation*}
(A+\epsilon)\log\log(n)\geq \frac{A}{1-\epsilon_1}\log\log(n)+\left[1-\frac{A}{1-\epsilon_1}\log\log(2)\right].
\end{equation*}
In particular, for $n\geq n_0$, we have
\begin{equation*}
\mu(S)\leq(A+\epsilon)\log\log(n)=\left[\frac{2}{\log(2)}+\epsilon\right]\log\log(n),
\end{equation*}
as claimed.
\end{proof}

\begin{prop}\label{prop:Ximu}
Let $n$ and $\mu$ be integers. Then,
\begin{equation*}
\Xi_\mu(n)\leq\binom{n-1}{\mu-1}\leq(n-1)^{\mu-1}.
\end{equation*}
\end{prop}
\begin{proof}
A semigroup $S$ of multiplicity $\mu$ can be described by its \emph{Apéry set} $\mathrm{Ap}(S,\mu):=\{0,a_1,\ldots,a_{\mu-1}\}$, where $a_i:=k_i\mu+i$ is the minimal element of $S$ congruent to $i$ modulo $\mu$ (see for example \cite[Chapter 1]{rosales_libro} for a deeper discussion of Apéry sets). In particular, it is uniquely described by the ordered sequence $(k_1,\ldots,k_{\mu-1})$.

Each $k_i$ is a positive integer (since there are no elements in $S$ smaller than $\mu$) and the sum $k_1+\cdots+k_{\mu-1}$ is equal to $\delta(S)$: indeed, if $x\in\insN\setminus S$ then $x=y_i\mu+i$, with $0\leq y_i<k_i$. The number of sequences $(k_1,\ldots,k_q)$ such that $k_1+\cdots+k_q\leq\delta$ is equal to the number of ordered partitions of $\delta+1$ into $q+1$ positive integers, or equivalently to the number of ways to divide a line of $\delta+1$ points into $q+1$ nonempty lines, which in turn is equal to the number of ways to place $q$ separators among $\delta$ holes; that is, it is equal to the number of subsets of $\{1,\ldots,\delta\}$ with $q$ elements, i.e., it is equal to $\binom{\delta}{q}$.

Since $|\insstar(S)|\geq\delta(S)+1$, we have our claim.
\end{proof}

We are ready to prove our best estimate.
\begin{teor}
For any $\epsilon>0$,
\begin{equation*}
\Xi(n)=O\left[\exp\left(\left(\frac{2}{\log(2)}+\epsilon\right)\log(n)\log\log(n)\right)\right].
\end{equation*}
\end{teor}
\begin{proof}
Let $A_\epsilon:=\frac{2}{\log(2)}+\epsilon$. For every $\epsilon$, and large enough $n$, we have $A_\epsilon\log\log(n)>4$; therefore, for large $n$,
\begin{equation*}
\Xi(n)=\sum_{\mu=3}^\infty\Xi_\mu(n)=\sum_{\mu=3}^{A_\epsilon\log\log(n)}\Xi_\mu(n).
\end{equation*}
Using Proposition \ref{prop:Ximu}, this becomes
\begin{equation*}
\Xi(n)\leq\sum_{\mu=3}^{A_\epsilon\log\log(n)}\Xi_\mu(n)\leq \sum_{\mu=3}^{A_\epsilon\log\log(n)}n^{\mu-1}\leq n^{A_\epsilon\log\log(n)}.
\end{equation*}

Since this holds for large $n$, the claim follows by writing $n^{A_\epsilon\log\log(n)}=\exp\left(A_\epsilon\log(n)\log\log(n)\right)$.
\end{proof}

\section{Acknowledgements}
The author wishes to thank the referee for his/her very careful reading of the manuscript and for his/her corrections.

\bibliographystyle{amsplain}

\begin{thebibliography}{10}

\bibitem{fontana_maximality}
Valentina Barucci, David~E. Dobbs, and Marco Fontana, \emph{Maximality
  properties in numerical semigroups and applications to one-dimensional
  analytically irreducible local domains}, Mem. Amer. Math. Soc. \textbf{125}
  (1997), no.~598, x+78.

\bibitem{fgh_semigruppi}
Ralf Fr{\"o}berg, Christian Gottlieb, and Roland H{\"a}ggkvist, \emph{On
  numerical semigroups}, Semigroup Forum \textbf{35} (1987), no.~1, 63--83.

\bibitem{gilmer}
Robert Gilmer, \emph{Multiplicative ideal theory}, Marcel Dekker Inc., New
  York, 1972, Pure and Applied Mathematics, No. 12.

\bibitem{twostar}
Evan~G. Houston, Abdeslam Mimouni, and Mi~Hee Park, \emph{Integral domains
  which admit at most two star operations}, Comm. Algebra \textbf{39} (2011),
  no.~5, 1907--1921.

\bibitem{houston_noeth-starfinite}
\bysame, \emph{Noetherian domains which admit only finitely many star
  operations}, J. Algebra \textbf{366} (2012), 78--93.

\bibitem{hmp_finite}
\bysame, \emph{Integrally closed domains with only finitely many star
  operations}, Comm. Algebra \textbf{42} (2014), no.~12, 5264--5286.

\bibitem{starnoeth_resinfinito}
Evan~G. Houston and Mi~Hee Park, \emph{A characterization of local noetherian
  domains which admit only finitely many star operations: The infinite residue
  field case}, J. Algebra \textbf{407} (2014), 105--134.

\bibitem{jager}
Joachim J{\"a}ger, \emph{L\"angenberechnung und kanonische {I}deale in
  eindimensionalen {R}ingen}, Arch. Math. (Basel) \textbf{29} (1977), no.~5,
  504--512.

\bibitem{star_semigroups}
Myeong~Og Kim, Dong~Je Kwak, and Young~Soo Park, \emph{Star-operations on
  semigroups}, Semigroup Forum \textbf{63} (2001), no.~2, 202--222.

\bibitem{kleitman_dedprob_II}
D.~Kleitman and G.~Markowsky, \emph{On {D}edekind's problem: the number of
  isotone {B}oolean functions. {II}}, Trans. Amer. Math. Soc. \textbf{213}
  (1975), 373--390.

\bibitem{krull_idealtheorie}
Wolfgang Krull, \emph{Idealtheorie}, Springer-Verlag, Berlin, 1935.

\bibitem{cardinality_pvd}
Mi~Hee Park, \emph{On the cardinality of star operations on a pseudo-valuation
  domain}, Rocky Mountain J. Math. \textbf{42} (2012), no.~6, 1939--1951.

\bibitem{rosales_libro}
José~Carlos Rosales and Pedro~A. Garc{\'{\i}}a-S{\'a}nchez, \emph{Numerical
  semigroups}, Developments in Mathematics, vol.~20, Springer, New York, 2009.

\bibitem{semigruppi_main}
Dario Spirito, \emph{Star {O}perations on {N}umerical {S}emigroups}, Comm.
  Algebra \textbf{43} (2015), no.~7, 2943--2963.

\bibitem{semigruppi_mu3}
\bysame, \emph{Star operations on numerical semigroups: the multiplicity 3
  case}, Semigroup Forum \textbf{91} (2015), no.~2, 476--494.

\bibitem{fibonacci_like}
Alex Zhai, \emph{Fibonacci-like growth of numerical semigroups of a given
  genus}, Semigroup Forum \textbf{86} (2013), no.~3, 634--662.

\end{thebibliography}
\providecommand{\bysame}{\leavevmode\hbox to3em{\hrulefill}\thinspace}
\providecommand{\MR}{\relax\ifhmode\unskip\space\fi MR }
\providecommand{\MRhref}[2]{%
  \href{http://www.ams.org/mathscinet-getitem?mr=#1}{#2}
}
\providecommand{\href}[2]{#2}

\end{document}